\documentclass[11pt,a4paper,reqno]{amsart}

\usepackage[a4paper,lmargin=2cm,rmargin=2cm,tmargin=4cm,bmargin=4cm]{geometry}

\usepackage[english]{babel}

\usepackage[centertags]{amsmath}
\usepackage{amsfonts}
\usepackage{amssymb}
\usepackage{amsthm}
\usepackage{dsfont}

\usepackage[dvipsnames]{xcolor}

\usepackage[normalem]{ulem}
\usepackage{bbm}

\usepackage{hyperref}
	\hypersetup{breaklinks=true,colorlinks=true,
linkcolor=MidnightBlue,citecolor=MidnightBlue,
urlcolor=MidnightBlue}
\usepackage{enumerate}

\usepackage{tikz-cd,graphicx}
\usepackage{tikz}
\usetikzlibrary{shapes.geometric, arrows}

\tikzstyle{startstop} = [rectangle, rounded corners, minimum width=2cm, minimum height=1cm,text centered,text width=2cm, draw=black]
\tikzstyle{io} = [rectangle, rounded corners, minimum width=2cm, minimum height=1cm,text centered,text width=2cm, draw=black]

\tikzstyle{arrow} = [thick,->,>=stealth]
\tikzstyle{arrow2} = [dashed,->,>=stealth]

\usepackage{dot2texi}

\numberwithin{equation}{section}
\usepackage{enumerate}
\usepackage{mathtools}

\usepackage[shortlabels]{enumitem}

\newcommand{\abs}[1]{\left\lvert#1\right\rvert}
\newcommand{\norm}[1]{\left\lVert#1\right\rVert}

\newcommand{\N}{\mathbb{N}}

\DeclareMathOperator{\dent}{dent}
\DeclareMathOperator{\Id}{Id}
\DeclareMathOperator{\diam}{diam}
\DeclareMathOperator{\Lip}{Lip}

\DeclareMathOperator{\conv}{conv}
\DeclareMathOperator{\SCD}{SCD}

\DeclareMathOperator{\real}{Re}

\DeclareMathOperator{\Span}{span}
\DeclareMathOperator{\ext}{ext}
\DeclareMathOperator{\strexp}{str-exp}
\DeclareMathOperator{\BCP}{BCP}

\DeclareMathOperator{\re}{Re\,}

\newcommand{\R}{\mathbb{R}}
\renewcommand{\geq}{\geqslant}
\renewcommand{\leq}{\leqslant}

\theoremstyle{plain}
 \newtheorem{theorem}{Theorem}[section]
 \newtheorem{lemma}[theorem]{Lemma}
 \newtheorem{proposition}[theorem]{Proposition}
 \newtheorem{corollary}[theorem]{Corollary}

\theoremstyle{definition}
 \newtheorem{definition}[theorem]{Definition}
 \newtheorem*{definition*}{Definition}
 \newtheorem{remark}[theorem]{Remark}
 \newtheorem{example}[theorem]{Example}

\newcommand{\pten}{\,\ensuremath{\widehat{\otimes}_\pi}\,}

\begin{document}

\title{Slicely countably determined points in Banach spaces}

\author[Langemets]{Johann Langemets}
\address[Langemets and L\~{o}o]{Institute of Mathematics and Statistics, University of Tartu, Narva mnt 18, 51009 Tartu, Estonia}
\email{johann.langemets@ut.ee}
\urladdr{\url{https://johannlangemets.wordpress.com/}}
\urladdr{
\href{https://orcid.org/0000-0001-9649-7282}{ORCID: \texttt{0000-0001-9649-7282} } }

\author[L\~{o}o]{Marcus L\~{o}o}
\email{marcus.loo@ut.ee}
\urladdr{
\href{https://orcid.org/0009-0003-1306-5639}{ORCID: \texttt{0009-0003-1306-5639} } }

\author[Mart\'in]{Miguel Mart\'in}
\address[Mart\'in and Rueda Zoca]{Universidad de Granada, Facultad de Ciencias. Departamento de Analisis Matem\'atico, 18071, Granada, Spain}
\email{mmartins@ugr.es}
\urladdr{\url{https://www.ugr.es/local/mmartins/}}
\urladdr{
\href{http://orcid.org/0000-0003-4502-798X}{ORCID: \texttt{0000-0003-4502-798X} } }

\author[Rueda Zoca]{Abraham Rueda Zoca}
\email{abraham.rueda@um.es}
\urladdr{\url{https://arzenglish.wordpress.com}}
\urladdr{
\href{https://orcid.org/0000-0003-0718-1353}{ORCID: \texttt{0000-0003-0718-1353} } }



\date{November 5th, 2023}

\maketitle

\begin{abstract}
We introduce slicely countably determined points (SCD points) of a bounded and convex subset of a Banach space which extends the notions of denting points, strongly regular points and much more. We completely characterize SCD points in the unit balls of $L_1$-preduals. We study SCD points in direct sums of Banach spaces and obtain that an infinite sum of Banach spaces may have an SCD point despite the fact that none of its components have it. We then prove sufficient conditions to get that an elementary tensor $x\otimes y$ is an SCD point in the unit ball of the projective tensor product $X \widehat{\otimes}_\pi Y$. Regarding Lipschitz-free spaces on complete metric spaces, we show that norm-one SCD points of their unit balls are exactly the ones that can be approximated by convex combinations of denting points of the unit ball. Finally, as applications, we prove a new inheritance result for the Daugavet property to its subspaces, we show that separable Banach spaces for which every convex series of slices intersects the unit sphere must contain an isomorphic copy of $\ell_1$, and we get pointwise conditions on an operator on a Banach space with the Daugavet property to satisfy the Daugavet equation.
\end{abstract}

\section{Introduction}

A well-investigated geometric property of Banach spaces is the Radon--Nikod\'ym property (RNP for short) because it emphasizes the interplay between the topological, geometrical, and measure theoretical structures of a Banach space. The RNP is known to have plentiful equivalent characterizations \cite[p.~217]{VectorMeasures}, e.g., a Banach space has the RNP if and only if every non-empty closed bounded convex subset is dentable.

A closely related and equally important geometric property of Banach spaces is Asplundness. A Banach space $X$ is Asplund if and only if every continuous convex real-valued function is Fr\'{e}chet differentiable on a dense set. It is important to recall that a Banach space is Asplund if and only if its dual space has the RNP.

In \cite{aviles_slicely_2010}, A.~Avil\'es, V.~Kadets, M.~Mart\'in, J.~Mer\'i, and V.~Shepelska introduced a (non-trivial) class of Banach spaces which contains both separable spaces with the RNP and separable Asplund spaces. Actually, the definition is given for bounded convex subsets as follows. A bounded convex subset $A$ of a (real or complex) Banach space $X$ is an \emph{SCD set} if there is a sequence $\{S_n\colon n\in\mathbb N\}$ of slices of $A$ such that $S\subseteq \overline{\conv}(B)$ whenever $B\subseteq A$ intersects all the $S_n$'s \cite[Definition~2.5]{aviles_slicely_2010}. Observe that every SCD set is clearly separable. It is proved in the same paper that the sequence of slices in the definition of SCD set can be replaced equivalently by a sequence of relative weakly open sets, allowing to get further examples. The space $X$ is called \emph{slicely countably determined} (\emph{SCD} for short) if every bounded convex subset $A$ of $X$ is SCD. Examples of SCD spaces are, on the one hand, those separable spaces with the RNP (actually, separable spaces with the convex point of continuity property, aka CPCP or even strongly regular spaces) and, on the other hand, separable Asplund spaces (actually, separable spaces which do not contain an isomorphic copy $\ell_1$) \cite[Examples~3.2]{aviles_slicely_2010}. The unit ball of every Banach space with a one-unconditional basis is SCD \cite[Theorem~3.1]{kadets_lushness_2013} and so is the unit ball of a locally uniformly rotund (aka LUR) separable Banach space \cite[Example~2.10]{aviles_slicely_2010}. It is an open question whether every Banach space with an unconditional basis is an SCD space.
On the contrary, the spaces $C[0,1]$ and $L_1[0,1]$ fail to be SCD as their unit balls fail to be SCD sets \cite[Example~2.13]{aviles_slicely_2010}. The same happens for all (separable) Banach spaces with the Daugavet property. Recall that a Banach space $X$ has the \emph{Daugavet property} \cite{kadets1997Daugavet} if every rank-one bounded linear operator $T$ on $X$ satisfies the norm equality
\begin{equation}\label{eq:DE}\tag{DE}
\|\Id + T\|=1+\|T\|
\end{equation}
where $\Id$ is the identity operator; in this case, all compact or weakly compact operators also satisfy the same norm equality \cite{kadets1997Daugavet}. Let us comment that these positive and negative examples show that being an SCD space is a non-trivial isomorphic property which covers the RNP and Asplundness (and much more!) in the separable case being, as far as we know, the first property of this kind. SCD sets and SCD spaces have been deeply used to get interesting results on the Daugavet property, numerical index one spaces, spear operators\ldots (see \cite{aviles_slicely_2010, KMMP-Spear, kadets_lushness_2013,kadets_operations_2018}, among other references).

The aim of this paper is to introduce and examine SCD points (see Definition~\ref{definition: SCD-point}) in order to study the SCD phenomena also in non-separable spaces. Equipped with this new pointwise view, we can even deduce new results about separable Banach spaces as well.

We present now the main contributions and the organization of the paper. Section~\ref{sec: SCD points} is devoted to introducing the concept of an SCD point, its equivalent formulations, first examples, and some preliminary properties, as the fact that the set of SCD points is always relatively closed and convex, and the relation between the SCD points of a set and those of its closure. In the second part of this section, we focus on SCD points of the unit ball. In particular, we prove that if the dual of a Banach space fails the $(-1)$-$\BCP$, then its unit ball has no SCD points (see Theorem~\ref{thm: sufficient condition to have empty set of SCD points}). This theorem is an extension of the result that unit balls of Daugavet spaces fail to have any SCD point (Example~\ref{newexample:Daguavet-SCD-empty}). As a consequence, we can characterize all the SCD points in an $L_1$-predual (see Theorem~\ref{thereom: L_1 preduals}).

In Section~\ref{sec: SCD points in direct sums}, we prove some stability results of SCD points by considering direct sums of Banach spaces. In particular, we will describe SCD points in the unit ball of $\ell_1$-sums and $\ell_\infty$-sums of Banach spaces (see Propositions \ref{proposition: infty sum equivalent condition} and \ref{proposition: 1-sum equivalent condition}). Additionally, we prove a result, which guarantees the existence of SCD points in the unit ball of any infinite $\ell_p$-sum of Banach spaces for $1<p<\infty$ (see Proposition~\ref{proposition: 0 is always SCD in infinite absolute sum}). This enables us, for instance, to find a separable Banach space having only one SCD point in its unit ball.

Section~\ref{sec: SCD points in tensor products} deals with SCD points in the unit ball of the projective tensor product of Banach spaces. We prove two sufficient conditions for an elementary tensor to be an SCD point in the unit ball (see Theorems~\ref{theorem: SCD tensor denting is SCD} and~\ref{theorem: a SCD b preserved extreme then a tensor b SCD}), and using these results, we answer some questions in the context of SCD sets as well.

Section~\ref{sec: SCD points in Lipschitz-free spaces} is dedicated to the characterization of SCD points in the unit ball of Lipschitz-free spaces. In particular, it is shown that, for complete metric spaces, norm-one SCD points of the unit ball are precisely the ones that can be approximated by convex combinations of denting points of the unit ball (see Theorem~\ref{theo:freecomplete}). Additionally, for compact metric spaces, a similar result holds when we consider denting points instead of strongly exposed points (see Theorem~\ref{theo:freecompact}). Using this method, we are able to partially describe SCD sets in Lipschitz-free spaces too.

In Section~\ref{sec: applications}, we provide applications of the theory developed for SCD points. Firstly, a new inheritance result for the Daugavet property to its subspaces is proved (see Theorem~\ref{theo:inheridaugascd}). Secondly, we are able prove that separable Banach spaces in which every convex series of slices of the unit ball intersects the unit sphere must contain an isomorphic copy of $\ell_1$ (see Theorem~\ref{thm: separable X contains ell1 whenever X* fails (-1)-BCP}). Finally, we also provide, using SCD points, a sufficient condition for a bounded linear operator on a Banach space with the Daugavet property to satisfy \eqref{eq:DE} (see Theorem~\ref{theorem: SCD points guarantee DE}).

We finish this introduction with some notation. Throughout the text, we will use standard notation, e.g.\ as in the textbook \cite{BST2011}. For a given Banach space $X$ over the field $\mathbb{K}$ ($\mathbb{R}$ or $\mathbb{C}$), we denote respectively by $B_X$ and $S_X$ the closed unit ball and the unit sphere of $X$, and we denote by $X^*$ the topological dual of $X$. Given two Banach spaces $X$ and $Y$, we denote by $\mathcal{L}(X, Y)$ the space of bounded linear operators $T\colon X \to Y$, endowed with the operator norm and we denote by $\mathcal{B}(X\times Y)$ the space of bounded bilinear maps $B\colon X \times Y \rightarrow \mathbb{K}$, endowed with its usual norm given by $$\|B\|=\sup\{|B(x,y)|\colon x\in B_X,\, y\in B_Y\}.$$

If $A$ is a non-empty subset of a Banach space $X$, we will denote respectively by $\conv(A)$, $\overline{\conv}(A)$, and $\diam(A)$ the convex hull, the closed convex hull, and the diameter of $A$. The set of extreme points of $A$ will be denoted as $\ext(A)$. By a \emph{slice} of $A$ we mean a non-empty intersection of $A$ with an open half-space, which can be always written in the form
\[
S(A,x^*,\alpha) := \left\{x\in A\colon \real x^*(x) > \sup_{b\in A}\real x^*(b)-\alpha\right\},
\]
for suitable $x^*\in X^*$ and $\alpha >0$. A \emph{convex combination of slices} (resp., \emph{convex series of slices}) of $A$ is a set of the form
\[
\sum_{i=1}^n \lambda_iS(A,x^*_i, \alpha_i) \qquad \bigg(\text{resp.}, \  \sum_{i=1}^{\infty} \lambda_iS(A,x^*_i, \alpha_i)\bigg),
\]
where $\lambda_1,\dots,\lambda_n\geq 0$ and $\sum_{i=1}^n \lambda_i = 1$ (resp., $\lambda_i\geq 0$ and $\sum_{i=1}^{\infty} \lambda_i = 1$).

Let $A\subseteq X$ be convex, and bounded. A point $x\in A$ is called a \emph{strongly exposed point} of $A$, if there exists $x^*\in X^*$ satisfying that for all sequences $(x_n)\subset A$ with $\re x^*(x_n)\rightarrow \re x^*(x)$, it follows that $x_n\rightarrow x$. A weaker version of a strongly exposed point is a denting point. An element $x\in A$ is called a \emph{denting point} of $A$, if for every $\varepsilon>0$ there exists a slice $S$ of $A$ such that $x\in S$ and $\diam(S)\leq \varepsilon$. We denote respectively by $\strexp(A)$ and $\dent(A)$ the set of all strongly exposed points and the set of all denting points of $A$.
Finally, recall that an extreme point of $A$ is called a \emph{preserved extreme point}, if it is an extreme point of the weak* closure of $A$ in $X^{**}$. It is well known that the open slices containing an preserved extreme point $x_0\in A$ form a basis of the weak topology of $A$ at $x_0$ \cite{LLT1988}.

Finally, let us give the notation for absolute sums of Banach spaces. If $X$ and $Y$ are Banach spaces and $N$ is an absolute norm on $\R^2$, we denote by $X\oplus_N Y$ the product space $X\times Y$ endowed with the norm $\|(x,y)\|=N(\|x\|,\|y\|)$ and we call it the ($N$-)\emph{absolute sum} of $X$ and $Y$. When $N$ is one of the $\ell_p$-norms in $\R^2$ (for $1\leq p\leq \infty$) we just write $X\oplus_p Y$. When dealing with a sequence of spaces $\{X_n\}_{n\in \N}$, given a Banach space of sequences  $(E,|\cdot|)$, we write
$$
\left[\bigoplus\nolimits_{n=1}^{\infty} X_n\right]_{E}
$$
to denote the $E$-sum of the sequence of spaces: those $(x_n)\in \prod\nolimits_{n=1}^\infty X_n$ such that
$$
\|(x_n)\|_E:=\bigl|\bigl(\|x_n||\bigr)\bigr|<\infty.
$$
In this way, $\left(\left[\bigoplus\nolimits_{n=1}^{\infty} X_n\right]_{E},\|\cdot\|_E\right)$ is a Banach space. The most interesting cases for us are $E=c_0$ and $E=\ell_p$ with $1\leq p \leq \infty$. For background on absolute norms and absolute sums, we refer the reader to \cite{Bonsall1,Hardtke,MPR2, MPR3}.

\section{Slicely countably determined points}\label{sec: SCD points}
We introduce here the main concept to deal in this manuscript: SCD points. We divide this preliminary section in two parts: first we give the definition, properties, and examples for general bounded convex sets; later, in a second subsection, we particularize the study for SCD points of the unit ball.

\subsection{Definition and first examples}
Firstly, we adapt the idea of a determining sequence of subsets to a point.

\begin{definition}\label{definition: determining family for point}
Let $X$ be a Banach space and $A\subseteq X$ bounded and convex. We say that a countable collection $\{V_n\colon n\in\mathbb{N}\}$ of (non-empty) subsets of $A$ is \textit{determining} for a point $a\in A$ if $a\in\overline{\conv}(B)$ whenever $B\subseteq A$ intersecting all the sets $V_n$.
\end{definition}

\begin{remark}\label{remark: replace convex set SCD-point}
{\slshape We can assume the set $B$ in Definition \ref{definition: determining family for point} to be convex.\ } Indeed, assume that we have $\{V_n\colon n\in\mathbb{N}\}$ such that for every convex subset $C\subseteq A$, satisfying $C\cap V_n\neq \emptyset$ for every $n\in\mathbb{N}$, we have $a\in \overline{\conv}(C)=\overline{C}$. Let $B\subseteq A$ be a set satisfying $B\cap V_n\neq \emptyset$ for every $n\in\mathbb{N}$. Then $\conv(B)\subseteq A$ and $\conv(B)$ also intersects every $V_n$, hence $a\in\overline{\conv}(\conv(B))=\overline{\conv}(B)$, meaning that the original condition also holds. The converse implication is evident.
\end{remark}

First we give some equivalent conditions for a sequence of subsets to be determining for a point. These characterizations will be used throughout the paper. Note that this discussion is similar to the one given in \cite{aviles_slicely_2010} for sequences determining sets, see \cite[Definition 2.1 and Proposition 2.2]{aviles_slicely_2010}.

\begin{proposition}\label{proposition: equivalent conditions of being determining for point}
Let $X$ be a Banach spaces, let $A\subseteq X$ be bounded and convex, and $a\in A$. For a sequence $\{V_n\colon n\in\mathbb{N}\}$ of subsets of $A$, the following conditions are equivalent:
\begin{enumerate}
    \item[\rm{(i)}] $\{V_n\colon n\in\mathbb{N}\}$ is determining for $a$;
    \item[\rm{(ii)}] for every slice $S$ of $A$ with $a\in S$, there is $m\in\mathbb{N}$ such that $V_m\subseteq S$;
    \item[\rm{(iii)}] if $x_n\in V_n$ for every $n\in \mathbb{N}$, then $a\in\overline{\conv}(\{x_n\colon n\in \mathbb{N}\})$.
\end{enumerate}
\end{proposition}

\begin{proof}
\parskip=0ex
    (i) $\Rightarrow$ (ii). Let $S$ be a slice of $A$ with $a\in S$. Assume on the contrary that $V_n\not\subseteq S$ for every $n\in\mathbb{N}$. Therefore, $(A\setminus S)\cap V_n\neq \emptyset$ for all $n\in\mathbb{N}$. By (i) $a\in \overline{\conv}(A\setminus S)=A\setminus S$, a contradiction.

    (ii) $\Rightarrow$ (iii). Let us prove the contrapositive. Assume that for every $n\in\mathbb{N}$ there exists $x_n\in V_n$ such that $a\not\in C:=\overline{\conv}(\{x_n\colon n\in\mathbb{N}\})$. By Hahn-Banach separation theorem, there exists a functional $x^*\in S_{X^*}$ such that
    \[
    \real x^*(a) > \underset{c\in C}{\sup}\real x^*(c).
    \]
    Let $\alpha := \real x^*(a)-\sup_{c\in C} \real x^*(c)>0$. Then, $a\in S(A,x^*,\alpha)$, but for arbitrary $n\in\mathbb{N}$ we get $V_n\not\subseteq S(A,x^*,\alpha)$, because $x_n\in V_n$ and
    \[
    \real x^*(x_n) \leq \underset{c\in C}{\sup}\real x^*(c) =  \real x^*(a) - \alpha,
    \]
    hence $x_n\not\in S(A,x^*,\alpha)$.

   (iii) $\Rightarrow$ (i). Immediate.
\end{proof}

The following observation is clear.

\begin{remark}\label{newremark:smaller sets are also determining}
{\slshape Let $X$ be a Banach spaces, let $A\subseteq X$ be bounded and convex, and $a\in A$. If $\{V_n\colon n\in\mathbb{N}\}$ is determining for $a$, then so is every countable collection $\{W_n\colon n\in \mathbb{N}\}$ such that $W_n\subset V_n$ for every $n\in \mathbb{N}$.}
\end{remark}

We can now give the main definition of the paper.

\begin{definition}\label{definition: SCD-point}
    Let $X$ be a Banach space and assume that $A\subseteq X$ is bounded and convex. A point $a\in A$ is called a \textit{slicely countably determined point of $A$} (an \textit{SCD point of $A$} for short), if there exists a determining sequence of slices of $A$ for the point $a$. We denote the set of all $\SCD$ points of $A$ by $\SCD(A)$.
\end{definition}

The following easy lemmata on the set of SCD points of a given bounded convex set will be very useful later on and will provide the first examples. The first result shows properties of the set $\SCD(A)$ for a given set $A$.

\begin{lemma}\label{lemma: properties of SCD(A)}
Let $X$ be a Banach space and assume that $A\subseteq X$ is bounded and convex. Then, $\SCD(A)$ is convex and closed (relative to $A$). Moreover, if $A$ is balanced, then so is $\SCD(A)$.
\end{lemma}

\begin{proof}
\parskip=0ex
Firstly, it is easy to see that $\SCD(A)$ is a (relative) closed set. Indeed, suppose that $(a_n)\subset \SCD(A)$ and $\lim_n a_n=a\in A$. For every $n\in \N$, there is a countable family $\{S_{n,m}\colon m\in \N\}$ which is determining for $a_n$. Now, the countable family
$$
\bigl\{S_{n,m}\colon m\in \N,\ n\in \N\}
$$
is determining for $a$: given a slice $S$ of $A$ containing $a$, as $S$ is relatively open in $A$, then $a_n\in S$ for some $n\in \N$ and, as $a_n\in \SCD(A)$, there is $m\in \N$ such that $S_{n,m}\subset S$; hence, $a\in \SCD(A)$.

For the convexity, fix $a,b\in \SCD(A)$ and $\lambda\in(0,1)$. Let us show that $\lambda a+(1-\lambda)b\in \SCD(A)$. By assumption there exist two determining sequences of slices for $a$ and for $b$. Consider the union of these sequences. If we have $B\subseteq A$ such that $B$ intersects all slices in the union, we have $a\in\overline{\conv}(B)$ and $b\in\overline{\conv}(B)$. It is clear that in this case $\lambda a+(1-\lambda)b\in \overline{\conv}(B)$. We have shown that $\SCD(A)$ is convex.

Assume finally that $A$ is balanced. We aim to show that $\SCD(A)$ is also balanced. By the convexity of $A$, it suffices to show that for every $a\in \SCD(A)$ and $|\lambda|=1$ one has that $\lambda a\in \SCD(A)$. Since $a\in \SCD(A)$, there is a determining sequence of slices $S_n:=S(A,x^*_n, \alpha_n)$ for $a$. It is routine to show that $T_n:=S(A, \overline{\lambda}x^*_n, \alpha_n)$ is a determining sequence of slices for $\lambda a$.
\end{proof}

The next lemma allows us to deal in some cases just with closed sets.

\begin{lemma}\label{newlemma:closure}
Let $X$ be a Banach space and let $A\subset X$ be convex and bounded. Then,
$$
\SCD\bigl(\overline{A}\bigr)\cap A=\SCD(A).
$$
\end{lemma}

\begin{proof}
Suppose that $a\in A\cap \SCD\bigl(\overline{A}\bigr)$. Then, there is a sequence of slices $\{S_n\}$ of $\overline{A}$ which is determining for $a$ in $\overline{A}$. Define $T_n=S_n\cap A$ for every $n\in \N$, which are (non-empty) slices of $A$, and let us prove that $\{T_n\}$ is determining for $a$ in $A$. Indeed, let $B\subset A$ such that $B\cap T_n\neq \emptyset$ for every $n\in \N$. Then, $B\subset \overline{A}$ clearly satisfies that $B\cap S_n\neq \emptyset$ for every $n\in \N$, hence $a\in \overline{\conv}(B)$, giving that $\{T_n\}$ is determining.

Conversely, suppose that $a\in\SCD(A)$ and let us show that $a\in \SCD\bigl(\overline{A}\bigr)$. Let $S_n:=S(A,x^*_n,\alpha_n)$ with $x^*_n\in X^*$, $\alpha_n>0$ for evert $n\in \N$ be a determining sequence for $a$ in $A$. Define for each $n\in \N$ the slice
$$
T_n=S\bigl(\overline{A},x^*_n,\alpha_n)
$$
which are slices of $\overline{A}$. Let now $B\subset \overline{A}$ satisfies that $B\cap
T_n=\emptyset$ for every $n\in \N$. Taking into account that $\sup_{A}\re x^*_n=\sup_{\overline{A}}\re x^*_n$ and that the slices are open, it follows that
$$
\bigl[A\cap (B + \tfrac{1}{m}B_X)\bigl] \cap S_n \neq \emptyset
$$
for every $m\in \N$. We get that $a\in \overline{\conv}\bigl(B+\tfrac{1}{m}B_X\bigr)$ for every $m\in\N$, hence $a\in \overline{\conv}(B)$.
\end{proof}

The next lemma describes a precise connection between SCD sets and SCD points, and confirms that Definition~\ref{definition: SCD-point} is indeed a natural extension.

\begin{lemma}\label{lemma: connection between SCD sets and SCD points}
Let $X$ be a Banach space and assume that $A\subseteq X$ is bounded and convex. Then, the following statements hold:
\begin{enumerate}
\item If $A$ is an $\SCD$ set, then every $a\in A$ is an $\SCD$ point.
\item If there is a countable dense subset of $A$ consisting in SCD points of $A$, then $A$ is SCD.
\item In particular, if every $a\in A$ is an $\SCD$ point and $A$ is separable, then $A$ is an $\SCD$ set.
\end{enumerate}
\end{lemma}

\begin{proof}
\parskip=0ex
(1) is clear by the corresponding definitions. Let us prove (2).
We need to find a determining sequence of slices for $A$. By assumption we have a countable set $\{x_n\colon n\in\mathbb{N}\}\subset \SCD(A)$, which is dense in $A$. Therefore, for every $x_n$ we have a determining sequence of slices $\{S_n^m\colon m\in\mathbb{N}\}$. Let us now consider the countable collection of slices $\{S_n^m \colon n,m\in\mathbb{N}\}$. To show that this collection is determining for $A$, we take an arbitrary slice $S$ of $A$. Since the set $\{x_n \colon n\in\mathbb{N}\}$ is dense, there exists $n\in\mathbb{N}$ such that $x_n\in S$ and then, by assumption, there also exists $m\in\mathbb{N}$ so that $S_n^m \subseteq S$.
\end{proof}

As a consequence, we get some more families of examples of SCD points by just using \cite[Examples~3.2]{aviles_slicely_2010} and Lemma~\ref{newlemma:closure}.

\begin{example}\label{example:examples of SCD sets}
{\slshape The following Banach spaces satisfy that $\SCD(A)=A$ for every convex bounded subset $A$:
\begin{enumerate}
  \item[(a)] separable Banach spaces with the RNP (or even with the CPCP, or strongly regular);
  \item[(b)] Asplund spaces (or even separable spaces which do not contain copies of $\ell_1$).
\end{enumerate}
}
\end{example}

On the other hand, it is immediate from the definition that denting points are SCD points, regardless of the separability or not of the set. Actually, an slightly more general definition than denting point also implies to be SCD point. Let $X$ be a Banach space and let $A$ be a bounded convex subset of $X$. An element $a\in A$ is said to be a \emph{quasi-denting} point of $A$ \cite{generating} if for every $\varepsilon>0$ there is a slice of $A$ contained in $a +\varepsilon B_X$. In particular, every denting point is quasi denting.

\begin{example}\label{example:dentingpoints}
{\slshape Let $X$ be a Banach space, let $A$ be a bounded convex subset of $X$. Then, every quasi-denting point of $A$ is SCD; in particular, $\dent(A)\subseteq \SCD(A)$.\ } Indeed, if $a\in A$ is quasi-denting, then there is a sequence $\{S_n\colon n\in\mathbb{N}\}$ of slices of $A$ such that $S_n\subset \bigl(a+\tfrac1n B_X\bigr)\cap A$. If $S$ is a slice of $A$ containing $a$, then $S$ is relatively open in $A$ and contains $a$, hence it has to contain $\bigl(a+\tfrac1n B_X\bigr)\cap A$ for some $n\in \mathbb{N}$. A fortiori, $S$ contains $S_n$.
\end{example}

As a consequence of the above example, Lemma~\ref{lemma: properties of SCD(A)}, and Lemma~\ref{newlemma:closure}, we get the following extension of Example~\ref{example:examples of SCD sets}.

\begin{example}\label{example:RNP}
{\slshape Let $X$ be a Banach space and let $A$ be bounded convex subset of $X$ such that $A\subset \overline{\conv}\bigl(\dent\bigl(\overline{A}\bigr)\bigr)$. Then, $\SCD(A)=A$. In particular, if $X$ has the RNP, then $\SCD(A)=A$ for every bounded convex subset $A$ of $X$.}
\end{example}

One might wonder whether in Condition (ii) of Proposition~\ref{proposition: equivalent conditions of being determining for point}, the determined point has to belong to all the sets determining it or, at least, to be close to one of them. The following example shows that each member of the determining sequence of slices can be far from the determined point. We need some notation. A point $x\in S_X$ of a Banach space $X$ is said to be a \emph{Daugavet point} \cite{abrahamsen_haller_lima_pirk_2020} if for every slice $S$ of $B_X$ and for every $\varepsilon>0$, there exists $y\in S$ such that $\|x-y\|\geq 2-\varepsilon$. Moreover, a Banach space $X$ has the Daugavet property if and only if every $x\in S_X$ is a Daugavet point (this is already contained in  \cite[Lemma~2.2]{kadets1997Daugavet}, see  \cite[Proposition~1.2]{abrahamsen_haller_lima_pirk_2020}).

\begin{example}
    {\slshape Let $X$ be a separable Banach space with the RNP and containing a Daugavet point $x_0\in S_X$ (a space like this is constructed in \cite[Example~3.1]{veeorg2021characterizations}). Then, for every $\varepsilon>0$, there exists a sequence of slices $\{S_n\colon n\in\mathbb{N}\}\subseteq B_X$ which is determining for $x_0$ such that $d(x_0, S_n)>2-\varepsilon$ for every $n\in\mathbb{N}$.}
\end{example}

\begin{proof}
\parskip=0ex
Fix $\varepsilon>0$. We have, by Lemma~\ref{lemma: connection between SCD sets and SCD points} or Example~\ref{example:RNP}, a determining sequence of slices $\{T_n\colon n\in\mathbb{N}\}$ for $x_0$. Since $B_X = \overline{\conv}(\dent(B_X))$, for every $n\in \N$ there exists $x_n\in T_n\cap \dent(B_X)$. This enables us to find a slice $S_n$ such that $x_n\in S_n\subseteq T_n$ and $\diam (S_n) < \varepsilon/2$ for every $n\in\mathbb{N}$. Observe that in this case the slices $\{S_n\colon n\in\mathbb{N}\}$ also determine the point $x_0$ (use Remark~\ref{newremark:smaller sets are also determining}), and since $x_0$ is a Daugavet point, we have $\norm{x_0-x_n}=2$ for every $n\in\mathbb{N}$ by \cite[Proposition~3.1]{jung2021daugavet}. In conclusion
        \[
        d(x_0,S_n) \geq \norm{x_0-x_n}-\diam(S_n) \geq 2- \frac{\varepsilon}{2}> 2-\varepsilon.\qedhere
        \]
\end{proof}

Our next aim is to prove that in the definition of a slicely countably determined point, one can equivalently replace slices with non-empty relatively weakly open sets or with convex combinations of slices.

\begin{proposition}\label{proposition: S_n iff W_n iff C_n}
Let $X$ be a Banach space and assume that $A\subseteq X$ is bounded and convex. Then, the following conditions are equivalent:
\begin{enumerate}
    \item[\rm{(i)}] $a \in \SCD(A)$;
    \item[\rm{(ii)}] there exists a sequence of relatively weakly open sets which is determining for $a$;
    \item[\rm{(iii)}] there exists a sequence of convex combinations of slices which is determining for $a$.
\end{enumerate}
\end{proposition}

\begin{proof}
\parskip=0ex
Firstly, note that (i) $\Rightarrow$ (ii) and (i) $\Rightarrow$ (iii) are obvious.

(ii) $\Rightarrow$ (iii). Assume that there exists determining sequence of relatively weakly open sets sets $\{W_n\colon n\in\mathbb{N}\}$ for the point $a$. By Bourgain's lemma \cite[Lemma~II.1]{GGMS1987}, every $W_n$ contains a convex combination of slices $C_n$ of $A$. But then the sequence $\{C_n\colon n\in \N\}$ is determining for $a$ (just use Remark~\ref{newremark:smaller sets are also determining}).

(iii) $\Rightarrow$ (i). Assume that there exists a sequence of convex combinations of slices $\{C_n\colon n\in\mathbb{N}\}$, which is determining for point $a$. For every $n\in\mathbb{N}$, we have
\begin{equation}\label{equation: definition of C_n}
C_n = \sum_{i=1}^{k_n}\lambda_i^n S_i^n,\quad \sum_{i=1}^{k_n}\lambda_i^n=1,\quad k_n\in\mathbb{N},
\end{equation}
where $S^n_i$, $i\in\{1,\dots,k_n\}$ are slices of $A$. Let us show that the countable family $\{S_i^n\colon n\in\mathbb{N},\; i\in\{1,\dots,k_n\}\}$ is determining for $a$. For that, let $B\subseteq A$ be convex such that $B\cap S_i^n\neq \emptyset$ for every $n\in\mathbb{N}$ and every $i\in\{1,\dots,k_n\}$. Hence there exists $b_i^n\in B\cap S^n_i$, from which we can construct another set
\[
\widehat{B} := \conv\{b^n_i\colon n\in\mathbb{N},\; i\in\{1,\dots,k_n\}\}\subseteq B\subseteq A.
\]
Let see that $\widehat{B}\cap C_n\neq \emptyset$ for every $n\in\mathbb{N}$. Indeed, by fixing $n\in\mathbb{N}$ arbitrarily, and using the fact that $b^n_i\in S^n_i$, we get
\[
\sum_{i=1}^{k_n}\lambda_i^n b_i^n \in \sum_{i=1}^{k_n} \lambda_i^n S_i^n = C_n
\]
and, since $\widehat{B}$ is convex, we have $\sum_{i=1}^{k_n}\lambda_i^n b_i^n\in \widehat{B}$. As $\{C_n\colon n\in \N\}$ is determining for $a$,  $a\in \overline{\conv}(\widehat{B})\subseteq \overline{\conv}(B)$.
\end{proof}

\begin{remark}
{\slshape
The claim in Remark \ref{remark: replace convex set SCD-point} holds also for determining sequences of relatively weakly open sets and convex combinations of slices sets, meaning we can assume convexity of the set $B$ in the definition of a determining sequence.}
\end{remark}

\begin{remark}
{\slshape With absolutely the same proof of (iii) $\Rightarrow$ (i) in Proposition~\ref{proposition: S_n iff W_n iff C_n}, one may show that $a\in \SCD(A)$ if there is a determining sequence of convex series of slices. Anyhow, we have not encounter any advantage in working with determining sequences of convex series of slices than working with determining sequences of convex combination of slices.}
\end{remark}

We may now get advantage of Proposition~\ref{proposition: S_n iff W_n iff C_n} to get more families of SCD points. Recall that a point $a$ of a closed convex bounded set $A$ is called a \emph{strongly regular point of A} if there exists a convex combination of slices of $A$ containing $a$ and having arbitrarily small diameter \cite[Remark in p.~29]{Rosenthal1988}. We will now prove that strongly regular points are SCD, hence also denting points (but this was already showed in Example~\ref{example:dentingpoints}).

\begin{proposition}\label{proposition: SCS is SCD}
Let $X$ be a Banach space and $A\subseteq X$ bounded and convex. If $a\in A$ satisfies that for every $\varepsilon>0$ there exists a convex combination $C$ of slices of $A$ such that $C\subseteq a+\varepsilon B_X$, then $a$ is an $\SCD$ point of $A$. In particular, strongly regular points and denting points are SCD points.
\end{proposition}

\begin{proof}
\parskip=0ex
For every $n\in \mathbb N$, we can find a convex combination $C_n$ of slices of $A$ such that $C_n\subseteq B(a,\frac{1}{n})$. By Proposition~\ref{proposition: S_n iff W_n iff C_n}, it suffices to show that the sequence $\{C_n\colon n\in\mathbb{N}\}$ is determining for $a$. Let $B\subseteq A$ be convex and $B\cap C_n\neq \emptyset$ for every $n\in\mathbb{N}$, hence there exists $b_n\in B\cap C_n$. We need to show that $a\in\overline{\conv}(B) = \overline{B}$. For $m\in \N$, we have
\[
\norm{a-b_m} \leq\frac{1}{m} < \varepsilon,
\]
therefore $a\in\overline{B}$.
\end{proof}

Recall that a closed convex bounded subset $A$ of a Banach space $X$ is said to be \emph{strongly regular} if every non-empty convex subset $L$ of $A$ contain convex combinations of slices of $L$ of arbitrary small diameter. $A$ is said to be a \emph{CPCP set} if for every non-empty convex subset $L$ of $A$ there exist relative weak open subsets of $L$ of arbitrary small diameter. CPCP sets are strongly regular (by Bourgain's lemma), but the reciprocal is not true. We refer to \cite{GGMS1987} for background. As a consequence of Proposition~\ref{proposition: SCS is SCD}, Lemmas \ref{lemma: properties of SCD(A)} and \ref{newlemma:closure}, and the fact that when $A$ is strongly regular, the set of strongly regular points of $A$ is norm dense \cite[Proposition~III.6]{GGMS1987}, we get the following.

\begin{example}\label{example:stronglyregular-set}
{\slshape Let $A$ be a convex bounded subset of a Banach space such that $\overline{A}$ is strongly regular (in particular, a CPCP set). Then, $\SCD(A)=A$.}
\end{example}

We now localize the concept of a (countable) $\pi$-base which was used in \cite[Proposition~2.21]{aviles_slicely_2010} in order to show that sets that do not contain $\ell_1$-sequences are SCD sets.

\begin{definition}
Let $X$ be a Banach space, $A\subseteq X$ bounded and convex, and $a\in A$. A \emph{local $\pi$-base} (for the weak topology of $A$) at $a$ is a family $\{W_j\colon j\in J\}\subseteq A$ of relatively weakly open subsets such that for every relatively weakly open set $W\subseteq A$ with $a\in W$, there exists $j\in J$ such that $W_j\subseteq W$.
\end{definition}

From the definition of a local $\pi$-base, the following observation is immediate.

\begin{lemma}
Let $X$ be a Banach space and $A\subseteq X$ bounded and convex. If $a\in A$ has a countable local $\pi$-base, then $a$ is an $\SCD$ point of $A$. Moreover, for every relatively weakly open set $W\subseteq A$ with $a\in W$, we have that $a\in \SCD(W)$.
\end{lemma}

For preserved extreme points, the above result is actually a characterization of being SCD point.

\begin{proposition}
If $a\in A$ is an $\SCD$ point and a preserved extreme point, then $a$ has a countable local $\pi$-base.
\end{proposition}

\begin{proof}
\parskip=0ex
Since $a$ is an $\SCD$ point there is a determining sequence $\{W_n\colon n\in \mathbb N\}\subseteq A$ of relatively weakly open subsets. Let $W\subseteq A$ be a relatively weakly open subset containing $a$. As $a$ is a preserved extreme point, we can find a slice $S$ of $A$ such that $a\in S\subseteq W$. Finally, there is $m\in \mathbb N$ such that $W_m\subseteq S\subseteq W$, because $a$ is an SCD point.
\end{proof}

We end this subsection by pointing out that we do not know whether there exist SCD points that do not have countable local $\pi$-bases.

\subsection{SCD points of the unit ball}
We start with an easy consequence of Lemma~\ref{lemma: properties of SCD(A)} which gives us a convenient way to check whether the unit ball of a Banach space has any SCD points.

\begin{corollary}\label{remark: 0 being SCD}
Let $X$ be a Banach space. Then, $\SCD(B_X)=\emptyset$ if and only if $0\not\in\SCD(B_X)$.
\end{corollary}

Our first aim is to investigate in detail when the unit ball of a Banach space fails to contain any SCD point. The first result in this line can be get by just having a sight to the proof of \cite[Example~2.13]{aviles_slicely_2010}: it is actually shown there that no point of the unit ball of a Banach space with the Daugavet property is a SCD point. Let us state the result for further reference.

\begin{example}\label{newexample:Daguavet-SCD-empty}
{\slshape Let $X$ be a Banach space with the Daugavet property. Then, $\SCD(B_X)=\emptyset$.}
\end{example}

We now want to extend the above result to a more general setting. For this, we recall a class of Banach spaces concerning the $(-1)$-ball covering property. We encourage the reader to consult \cite{ciaci2022characterization} and \cite{BCPInitial} for further reading.

\begin{definition}[{\cite[Definition~2.3]{ciaci2022characterization}}] \label{proposition: characterization of failing -1-BCP}
A Banach space $X$ is said to \textit{fail $(-1)$-ball covering property} (\emph{fail $(-1)$-$\BCP$} for short) if for every separable subspace $Y$ of $X$, there exists $x\in S_X$ such that the equality
\begin{equation}\label{equation: -1-BCP condition}
\norm{y+\lambda x} = \norm{y} + \abs{\lambda}
\end{equation}
holds for every $y\in Y$ and $\lambda \in \mathbb{R}$.
\end{definition}

Examples of Banach spaces failing the $(-1)$-$\BCP$ include $\ell_1(I)$, where $I$ is an uncountable set, the space $\ell_{\infty}/c_0$, and $X^*$ whenever $X$ has the Daugavet property \cite{ciaci2022characterization}.

The following result generalizes Example~\ref{newexample:Daguavet-SCD-empty}.

\begin{theorem}\label{thm: sufficient condition to have empty set of SCD points}
If $X$ is a Banach space such that $X^*$ fails $(-1)$-$\BCP$, then $\SCD(B_X)=\emptyset$.
\end{theorem}

\begin{proof}
\parskip=0ex
By Corollary~\ref{remark: 0 being SCD}, it suffices to show that $0\notin \SCD(B_X)$. Pick an arbitrary sequence of slices $\{S_n\colon n\in\mathbb{N}\}$, defined as $S_n = S(B_X,x^*_n,\alpha_n)$, where $x^*_n\in S_{X^*}$ and $\alpha_n>0$. We aim to show that there are $x_n\in S_n$ such that $0\not\in\overline{\conv}(\{x_n\colon n\in\mathbb{N}\})$. Since $X^*$ fails $(-1)$-BCP, we can find $x^*\in S_{X^*}$ such that
\[
\norm{x^*_n+x^*} = \norm{x^*_n} + 1 = 2
\]
for every $n\in\mathbb{N}$. By this condition, we can find for each $n\in\mathbb{N}$ an element $x_n\in S_{X}$ such that
\[
\real [x^*_n+x^*](x_n) > 2- \min\Big\{\alpha_n, \frac{1}{2}\Big\}.
\]
From this we infer that $x_n\in S_n$ for every $n\in\mathbb{N}$, because
\[
\real x^*_n(x_n)+1 \geq \real x_n^*(x_n) + \real x^*(x_n) > 2- \min\Big\{\alpha_n, \frac{1}{2}\Big\} \geq 2- \alpha_n.
\]
Similarly, we see that $\real x^*(x_n)> 1/2$ for every $n\in\mathbb{N}$,  hence $0\not\in\overline{\conv}(\{x_n\colon n\in\mathbb{N}\})$. This shows that $0\notin \SCD(B_X)$.
\end{proof}

Theorem~\ref{thm: sufficient condition to have empty set of SCD points} helps us to construct new (even separable) Banach spaces whose unit balls have no SCD points. Indeed, spaces whose duals fail $(-1)$-$\BCP$ include $C(K)$ spaces whenever $K$ is a compact Hausdorff space such that $|K|\geq \omega_1$ \cite[Theorem~6.3]{CIACI2022126185} and $L_1(\mu)$ spaces whenever $\mu$ is an atomless measure \cite[Corollary~6.7]{CIACI2022126185}. Let $X$ be one of such spaces and consider its direct sum $X\oplus_N X$ with the norm $N$ as suggested in \cite[Remark~6.8]{CIACI2022126185}. Then, the dual of $X\oplus_N X$ fails $(-1)$-$\BCP$ and so the unit ball of $X\oplus_N X$ has no SCD points by Theorem~\ref{thm: sufficient condition to have empty set of SCD points}; on the other hand, $X\oplus_N X$ fails the Daugavet property.

We end this section by studying SCD points in the unit ball of $L_1$-preduals. Recall that a Banach space $X$ is an $L_1$-predual if $X^*=L_1(S,\Sigma,\mu)$ for some measure space $(S,\Sigma,\mu)$. We need some notation. Given a Banach space $X$, we consider the equivalence relation $f\sim g$ if and only if $f$ and $g$ are linearly dependent elements of $\ext(B_{X*})$. We denote the quotient set by $\ext(B_{X*})/\sim$.

\begin{theorem}\label{thereom: L_1 preduals}
Let $X$ be an $L_1$-predual. Then, the following statements hold:
\begin{enumerate}
    \item [(a)] If $\ext(B_{X^*})/\sim$ is at most countable, then $\SCD(B_X) = B_X$;
    \item [(b)] If $\ext(B_{X^*})/\sim$ is uncountable, then $\SCD(B_X) = \emptyset$.
\end{enumerate}
\end{theorem}

To give the proof, we need to recall the well known decomposition of every $L_1(\mu)$-space in the form
\begin{equation}\label{equation: decomposition of L_1(mu)}
L_1(\mu) = L_1(\nu)\oplus_1 \ell_1(I),
\end{equation}
where $\nu$ is an atomless measure and $I$ is some index set (see \cite[Theorem~2.1]{measureJohnson}, for instance).

\begin{proof}
\parskip=0ex
    $(a)$. If $\ext(B_{X^*})/\sim$ is countable, then it is known that $X^*$ is separable \cite[Theorem~2]{Fonf-1978}, \cite[Theorem~3.1]{Fonf-Phelps1968}.
    Thus, $X$ is Asplund and separable, hence $X$ is an SCD space by \cite[Example~3.2]{aviles_slicely_2010}, in particular, $\SCD(B_X) = B_X$.

    $(b)$. Assume that $\ext(B_{X^*})/\sim$ is uncountable. Observe that the only extreme points of $B_{L_1(\mu)}$ are those in the second factor of the decomposition (\ref{equation: decomposition of L_1(mu)}) (use \cite[Lemma~I.1.5]{HWW} and the fact that the unit ball of $L_1(\nu)$ has no extreme points). Besides, $\ext(B_{\ell_1(I)}) = \{\lambda e_i\colon i\in I, |\lambda|=1\}$, where $e_i(j)=\delta_{ij}$. We deduce that $I$ is uncountable. It is known that in this case, $\ell_1(I)$ fails $(-1)$-BCP \cite[Corollary~25]{BCPInitial} and so does the absolute sum $L_1(\nu)\oplus_1 \ell_1(I)$ by \cite[Proposition~8]{BCPInitial}. Theorem~\ref{thm: sufficient condition to have empty set of SCD points} then shows that $\SCD(B_X)=\emptyset$.
\end{proof}

A consequence of the above result is the following interesting example.

\begin{example}
{\slshape Let $I$ be an uncountable set. Then, $\SCD(B_{c_0(I)})=\emptyset$.\ }
\end{example}

It is interesting here that $c_0(I)$ is Asplund (hence, in particular, it does not contain copies of $\ell_1$), hence Example~\ref{example:examples of SCD sets}.(b) does not extend to the non-separable case. Compare with the case of item (a) of the same Example, which extends to the non-separable case, see Example~\ref{example:stronglyregular-set}. We do not know which could be the version of the SCD property in the non-separable case which covers Asplund spaces.

\section{SCD points in absolute sums}\label{sec: SCD points in direct sums}
Our aim here is to present various stability results for SCD points by absolute sums and also to construct some interesting examples using this techniques.

\subsection{SCD points in the sum $X \oplus_\infty Y$}

Recall that for Banach spaces $X$ and $Y$ one has that $B_{X\oplus_{\infty} Y} = B_X\times B_Y$. Thus the following two preliminary results are straightforward adaptations of \cite[Lemmata~2.3~and~2.4]{kadets_operations_2018} where sums and unions of SCD sets were studied.

\begin{lemma}[cf.\ {\cite[Lemma~2.3]{kadets_operations_2018}}]\label{lemma: slice inclusion in l infinity sum}
Let $X$ and $Y$ be Banach spaces. For any $(x^*,y^*)\in (X\oplus_{\infty} Y)^*$ and $ \alpha, \beta, \gamma>0$, the following conditions hold:
\begin{enumerate}
    \item[(a)] $S(B_X, x^*, \alpha)\times S(B_Y, y^*,\beta) \subseteq S(B_{X \oplus_\infty Y}, (x^*,y^*),\alpha + \beta)$.
    \item[(b)] If $a\in B_X$, $b\in B_Y$ satisfy $(a,b)\in S(B_{X \oplus_\infty Y}, (x^*,y^*), \gamma)$, then
    $a\in S(B_X, x^*,\gamma)$ and $b\in S(B_Y, y^*,\gamma)$.
\end{enumerate}
\end{lemma}

\begin{proof}
\parskip=0ex
    $(a)$. Pick $(x,y)\in S(B_X, x^*, \alpha)\times S(B_Y, y^*,\beta)$. Then
    \begin{align*}
        &\real x^*(x)>\norm{x^*}-\alpha, \; \real y^*(y)>\norm{y^*}-\beta \\
        &\implies \real (x^*,y^*)(x,y) > \norm{x^*}+ \norm{y^*} - (\alpha + \beta)\\
        &\iff \real (x^*,y^*)(x,y) > \norm{(x^*,y^*)}_1 -(\alpha + \beta),
    \end{align*}
    hence $(x,y)\in S(B_{X\oplus_{\infty} Y}, (x^*,y^*),\alpha + \beta)$.

    $(b)$. Let $a\in B_X$, $b\in B_Y$ be such that $(a,b)\in S(B_Z,(x^*,y^*),\gamma)$. This means that
    \begin{align*}
        \real x^*(a)+\norm{y^*} \geq \real x^*(a) + \real y^*(b) > \norm{(x^*,y^*)}_1-\gamma = \norm{x^*} + \norm{y^*} - \gamma,
    \end{align*}
    from which we can conclude that $\real x^*(a)> \norm{x^*}-\gamma$, therefore $a\in S(B_X,x^*,\gamma)$. The proof for $b\in S(B_Y, y^*,\gamma)$ is analogous.
\end{proof}

The next result is an adaptation of \cite[Lemma~2.4]{kadets_operations_2018} to the new pointwise setting. We include the proof for the sake of completeness.

\begin{lemma}\label{lemma: necessary condition for pair to be SCD-point}
Let $X$ and $Y$ be Banach spaces. Assume that $(a,b)\in \SCD(B_{X\oplus_{\infty} Y})$. Then, there exists a  sequence $((x^*_n,y^*_n),\alpha_n)\in (X\oplus_{\infty} Y)^*\times (0,\infty)$ such that for every $(x^*,y^*)\in (X\oplus_{\infty} Y)^*$ and $\alpha>0$ satisfying $(a,b)\in S(B_{X\oplus_{\infty} Y}, (x^*,y^*),\alpha)$, there exists $m\in\mathbb{N}$ such that
\begin{equation}\label{lemma necessary condition inclusions}
S(B_X, x^*_m, \alpha_m)\subseteq S(B_X, x^*,\alpha) \;\text{ and }\;  S(B_Y, y^*_m, \alpha_m)\subseteq S(B_Y, y^*,\alpha).
\end{equation}
\end{lemma}

\begin{proof}
\parskip=0ex
    Let $S_n = S(B_{X\oplus_{\infty} Y}, (x^*_n, y^*_n), 2\alpha_n)$ be the determining sequence of slices for $(a,b)\in B_{X\oplus_{\infty} Y}$. We show that the desired sequence is $((x^*_n,y^*_n), \alpha_n)$, where $n\in\mathbb{N}$. Pick $(x^*,y^*)\in (X\oplus_{\infty} Y)^*$, $\alpha>0$ so that $(a,b)\in S(B_{X\oplus_{\infty} Y}, (x^*,y^*),\alpha)$. Since the sequence $\{S_n\colon n\in\mathbb{N}\}$ is determining for $(a,b)$, we can find $m\in\mathbb{N}$ such that
    \[
    S(B_{X\oplus_{\infty} Y}, (x^*_m,y^*_m), 2\alpha_m) \subseteq S(B_{X\oplus_{\infty} Y}, (x^*,y^*),\alpha).
    \]
    An application of Lemma \ref{lemma: slice inclusion in l infinity sum}.(a) gives us
    \[
    S(B_X, x^*_m, \alpha_m)\times S(B_Y, y^*_m,\alpha_m) \subseteq S(B_{X\oplus_{\infty} Y}, (x^*,y^*),\alpha).
    \]
    Observe that inclusions (\ref{lemma necessary condition inclusions}) are satisfied. Indeed, making use of Lemma \ref{lemma: slice inclusion in l infinity sum}.(b):
    \begin{align*}
        (x,y)\in S(B_X, x^*_m, \alpha_m)\times S(B_Y, y^*_m, \alpha_m) &\implies (x,y)\in S(B_{X\oplus_{\infty} Y}, (x^*,y^*),\alpha)\\
        &\implies x\in S(B_X,x^*,\alpha), \; y\in S(B_Y,y^*,\alpha),
    \end{align*}
    which implies that $ (x,y) \in S(B_X,x^*,\alpha)\times S(B_Y, y^*,\alpha)$.
\end{proof}

With the use of the presented lemmata, we can completely characterize the SCD points of the unit ball of an $\ell_\infty$ sum of two spaces.

\begin{proposition}\label{proposition: infty sum equivalent condition}
Let $X$ and $Y$ be Banach spaces. An element $(a,b)\in \SCD(B_{X\oplus_{\infty} Y})$ if and only if $a\in\SCD(B_X)$ and $b\in\SCD(B_Y)$.
\end{proposition}

\begin{proof}
\parskip=0ex
\textit{Necessity}. Assume that $(a,b)\in \SCD(B_{X\oplus_{\infty} Y})$
and let
\[((x^*_n,y^*_n), \alpha_n)\in (X\oplus_{\infty} Y)^*\times (0,\infty)\]
be the sequence from Lemma \ref{lemma: necessary condition for pair to be SCD-point}. Let us show that slices $S_n = S(B_X,x^*_n, \alpha_n)$ are determining for $a$. Fix a slice $S(B_X, x^*,\alpha)\ni a$. Observe that $(x^*,0)\in (X\oplus_{\infty} Y)^*$ and $(a,b)\in S(B_{X\oplus_{\infty} Y}, (x^*,0), \alpha)$, hence, by Lemma \ref{lemma: necessary condition for pair to be SCD-point}, there exists $m\in\mathbb{N}$ such that
\[
S_m = S(B_X,x^*_m,\alpha_m)\subseteq S(B_X, x^*,\alpha),
\]
so $a\in\SCD(B_X)$. Analogously we deduce that $b\in \SCD(B_Y)$.

\textit{Sufficiency}. Assume that sequences $\{S_n^a \colon n\in\mathbb{N}\}$ and $\{S_n^b \colon n\in\mathbb{N}\}$ determine points $a$ and $b$ respectively. We aim to prove, by using Proposition \ref{proposition: S_n iff W_n iff C_n} Condition (ii), that the sequence of non-empty relatively weakly open subsets $\{S^a_n\times S^b_k\colon n,k\in\mathbb{N}\}$ determines point $(a,b)$. Fix a slice $S=S(B_{X\oplus_{\infty} Y}, (x^*,y^*), \alpha)$ containing $(a,b)$, where $\norm{(x^*,y^*)}_1 = 1$ and $\alpha>0$. See that in this case
 \[
 \alpha > 1-\real x^*(a) - \real  y^*(b),
 \]
 which enables us to find $\gamma>0$ satisfying the inequality
 \[
 \alpha > 1-\real x^*(a) - \real  y^*(b) + \gamma.
 \]
 Notice that
 \[
\real x^*(a) > \norm{x^*} - \Big(\norm{x^*}-\real x^*(a) + \frac{\gamma}{2}\Big), \quad \real y^*(b) > \norm{y^*} - \Big(\norm{y^*}-\real y^*(b) + \frac{\gamma}{2}\Big).
\]
From this we infer that
\[
a\in S_1:=S\Big(B_X,x^*,\norm{x^*}-\real x^*(a) + \frac{\gamma}{2}\Big)\]
and
\[ b\in S_2:=S\Big(B_Y,y^*,\norm{y^*}-\real y^*(b) - \frac{\gamma}{2}\Big),
\]
hence it is possible to find $i,j\in \mathbb{N}$ such that $S^a_i\subseteq S_1$ and $S^b_j \subseteq S_2$ (as $a\in \SCD(B_X)$ and $b\in \SCD(B_Y)$). Now, making use of Lemma \ref{lemma necessary condition inclusions}.(a), we get
$S_1 \times S_2 \subseteq S$, because
\begin{align*}
\Big(\norm{x^*}-\real x^*(a) +\frac{\gamma}{2}\Big) &+ \Big(\norm{y^*}-\real y^*(b) +\frac{\gamma}{2}\Big) \\
&= \norm{(x^*,y^*)}_1 - \real x^*(a)-\real y^*(b) + \gamma \\
&= 1 -\real (x^*,y^*)(a,b) +\gamma < \alpha.
\end{align*}
In conclusion, we have $S^a_i\times S^b_j \subseteq S_1\times S_2\subseteq S$, which means that the countable collection of non-empty relatively weakly open subsets $\{S_n^a\times S_k^b\colon n,k\in\mathbb{N}\}$ determines point $(a,b)$.
\end{proof}

\subsection{SCD points in the sum $X \oplus_p Y$, where $1\leq p <\infty$.}
Our first result here studied some SCD points of the unit ball of an $\ell_p$-sum of two spaces for $1\leq p<\infty$.

\begin{proposition}\label{prop: SCD points in finite p-sum}
    Let $X$ and $Y$ be Banach spaces and $1\leq p<\infty$.
    \begin{enumerate}
        \item[(a)] If $a\in \SCD(B_X)$, then $(a,0)\in \SCD(B_{X\oplus_p Y})$.
        \item[(b)] If $a\in S_X$ and $(a,0)\in \SCD(B_{X\oplus_p Y})$, then $\lambda a\in \SCD(B_X)$ for every $\lambda\in[-1,1]$.
    \end{enumerate}
\end{proposition}

\begin{proof}
\parskip=0ex
    (a). Suppose that $a\in\SCD(B_X)$, with the determining sequence of slices $\{S_n\colon n\in\mathbb{N}\}$, where $S_n = S(B_X, x^*_n, \alpha_n)$, $x^*_n\in S_{X^*}$ for every $n\in\mathbb{N}$. Without loss of generality we can assume that $\alpha_n< 1$ for every $n\in\mathbb{N}$.

    Define for every $n,k\in\mathbb{N}$ a slice $S_n^k = S(B_{X\oplus_p Y}, (x^*_n,0), \alpha_n/k)$. We show that the countable collection of slices $\{S_n^k\colon n,k\in\mathbb{N}\}$ is determining for point $(a,0)$. Fix elements $z_n^k = (x^k_n,y^k_n)\in S_n^k$ for every $n,k\in\mathbb{N}$. It is easy to see that $x^k_n \in S_n$ and since the sequence $\{S_n\colon n\in\mathbb{N}\}$ determines point $a$, we have $a\in\overline{\conv}(\{x_n^k\colon n\in\mathbb{N}\})$ for every $k\in\mathbb{N}$. Furthermore, we see that
    \[
    1\geq \norm{(x^k_n,y^k_n)}^p_p = \norm{x_n^k}^p + \norm{y_n^k}^p > \left(1- \frac{\alpha_n}{k}\right)^p + \norm{y^k_n}^p>1-\frac{p\alpha_n}{k}+\norm{y^k_n}^p,
    \]
    from which
    \[
    \norm{y_n^k}^p< \frac{p\alpha_n}{k} < \frac{p}{k}.
    \]
    Pick an arbitrary $\varepsilon>0$ and find $K\in\mathbb{N}$ such that $p/K < \varepsilon/2$. By the argument presented before, we can find $z\in \conv(\{x_n^K\colon n\in\mathbb{N}\})$ such that $\norm{z-a}< (\varepsilon/2)^{1/p}$. Let $z=\sum_{n=1}^{\infty}\lambda_n x^K_n$, where $\sum_{n=1}^{\infty} \lambda_n=1$, $\lambda_n \in [0,1]$ and the number of non-zero elements $\lambda_n$ is finite. Now
    \begin{align*}
        \norm{(a,0)- \sum_{n=1}^{\infty}\lambda_n (x^K_n,y^K_n)}^p_p  &= \norm{\Big(a-\sum_{n=1}^{\infty}\lambda_n x^K_n, - \sum_{n=1}^{\infty}\lambda_n y^K_n\Big)}^p_p\\
        &= \norm{a-\sum_{n=1}^{\infty}\lambda_n x^K_n}^p + \norm{\sum_{n=1}^{\infty}\lambda_n y^K_n}^p\\
        &< \frac{\varepsilon}{2} + \sum_{n=1}^{\infty}\lambda_n \norm{y^K_n}^p
        < \frac{\varepsilon}{2} + \sum_{n=1}^{\infty}\lambda_n \cdot\frac{p}{K}\\
        &= \frac{\varepsilon}{2}+ \frac{p}{K}\sum_{n=1}^{\infty} \lambda_n \\
        &\leq \frac{\varepsilon}{2}+ \frac{p}{K}< \frac{\varepsilon}{2}+\frac{\varepsilon}{2}=\varepsilon,
    \end{align*}
    hence $(a,0)\in\overline{\conv}(\{z^K_n\colon n\in\mathbb{N}\})\subseteq\overline{\conv}(\{z^k_n\colon n,k\in\mathbb{N}\})$.

(b). Observe that it suffices to prove that $a\in \SCD(B_X)$ by Lemma~\ref{lemma: properties of SCD(A)}. Consider a sequence of slices $\{S_n\colon n\in\mathbb{N}\}$ of the form
    \[
    S_n = S(B_{X\oplus_p Y}, (x^*_n,y^*_n),\alpha_n),
    \]
where $\alpha_n>0$ and $\norm{(x^*_n,y^*_n)}_{q}=1$ for all $n\in\mathbb{N}$, which is determining for the point $(a,0)$.
Let us show that the sequence of slices $\{T_n\colon n\in\mathbb{N}\}$, where $T_n = S(B_X, x^*_n,\alpha_n)$, determines the point $a$. Pick for every $n\in\mathbb{N}$ an element $x_n\in T_n$. Our goal is to show that $a\in\overline{\conv}(\{x_n\colon n\in\mathbb{N}\})$. To do so, denote
    \[
    I:= \{n\in\mathbb{N}\colon \norm{x^*_n}\neq 0\}, \; J:=\mathbb{N}\setminus I,
    \]
and define elements $z_n\in B_{X\oplus_p Y}$ as follows:
    \[
    z_n=
    \begin{cases}
    (x_n,0), \; n\in I;\\
    (0,y_n), \; n\in J,
    \end{cases}
    \]
where for every $n\in\mathbb{N}$ the element $y_n$ is picked such that $y_n\in S(B_Y, y^*_n,\alpha_n)$. It is easy to see that every $z_n$ belongs to the slice $S_n$ and so, since the sequence $\{S_n\colon n\in\mathbb{N}\}$ is determining for $(a,0)$, we have $(a,0)\in\overline{\conv}(\{z_n\colon n\in\mathbb{N}\})$.

Fix $\varepsilon>0$ and find $\lambda_n\in [0,1]$ for every $n\in\mathbb{N}$, where finitely many of $\lambda_n$ are non-zero and $\sum_{n=1}^{\infty}\lambda_n =1$, such that
    \begin{equation}\label{close convex combination proposition 2.19}
    \norm{(a,0)-\sum_{n=1}^{\infty}\lambda_n z_n}^p_p< \Bigg(\frac{\varepsilon}{2}\Bigg)^p.
    \end{equation}
Then, we have
    \begin{align*}
    \norm{(a,0)-\sum_{n=1}^{\infty}\lambda_n z_n}^p_p &=\norm{(a,0)- \Big(\sum_{n\in I}\lambda_nx_n,0\Big) - \Big(0,\sum_{n\in J}\lambda_ny_n\Big)}^p_p\\
    &=\norm{\Big(a-\sum_{n\in I}\lambda_n x_n, -\sum_{n\in J} \lambda_n y_n\Big)}^p_p\\
    &= \norm{a-\sum_{n\in I}\lambda_n x_n}^p+ \norm{\sum_{n\in J} \lambda_n y_n}^p,
    \end{align*}
    and using the estimation (\ref{close convex combination proposition 2.19}), we get
    \[
    \norm{a-\sum_{n\in I}\lambda_n x_n} < \frac{\varepsilon}{2}.
    \]
    Notice that using the reverse triangle inequality
    \[
    \frac{\varepsilon}{2}> \norm{a-\sum_{n\in I}\lambda_n x_n} \geq \abs{\norm{a}-\norm{\sum_{n\in I}\lambda_n x_n}} = \abs{1-\norm{\sum_{n\in I}\lambda_n x_n}},
    \]
    therefore
    \[
    1-\frac{\varepsilon}{2}<\norm{\sum_{n\in I}\lambda_n x_n}\leq \sum_{n\in I}\lambda_n \leq 1.
    \]
    Denote $\Lambda := \sum_{n\in I}\lambda_n$ and see that $1-\varepsilon/2< \Lambda \leq 1$.

    Set \[
    x:=(1-\Lambda)x_K+\sum_{n\in I}\lambda_n x_n
    \]
    for some $K\in\mathbb{N}$ and notice that $x\in \conv(\{x_n\colon n\in \mathbb{N}\})$. Finally,
    \begin{align*}
        \|a-x\|\leq \norm{a-\sum_{n\in I}\lambda_n x_n}+(1-\Lambda)\|x_K\|<\frac{\varepsilon}{2} + \frac{\varepsilon}{2} = \varepsilon.
    \end{align*}
    Hence $a\in\overline{\conv}(\{x_n\colon n\in I\})\subseteq\overline{\conv}(\{x_n\colon n\in \mathbb{N}\}) $.
\end{proof}

We can now partially characterize the SCD points of the unit sphere of $X\oplus_1 Y$.

\begin{proposition}\label{proposition: 1-sum equivalent condition}
    Let $X$ and $Y$ be Banach spaces and $(a,b)\in S_{X\oplus_1Y}$, where $a\in X\setminus\{0\}$ and $b\in Y\setminus\{0\}$. Then, $(a,b)\in \SCD(B_{X\oplus_1 Y})$ if and only if $\frac{a}{\norm{a}}\in \SCD(B_X)$ and $\frac{b}{\norm{b}} \in \SCD(B_Y)$.
    \end{proposition}

    \begin{proof}
\parskip=0ex
        \textit{Necessity}. Let the slices $S_n =  S(B_{X\oplus_1 Y}, (x^*_n,y^*_n),\alpha_n)$, $n\in\mathbb{N}$, determine the point $(a,b)$. We prove that the slices $\{T_n\colon n\in\mathbb{N}\}$, where $T_n = S(B_X, x^*_n,\alpha_n)$, are determining for the point $\frac{a}{\norm{a}}$. Pick for each $n\in\mathbb{N}$ an element $x_n\in T_n$. We assume that $\norm{(x^*_n,y^*_n)}_{\infty} = \max\{\norm{x^*_n}, \norm{y^*_n}\}$=1, and using this fact, we construct the sets
        \[
        I:=\{n\in\mathbb{N}\colon \norm{x^*_n}=1\},\quad J:= \mathbb{N}\setminus I=\{n\in\mathbb{N}\colon \norm{y^*_n}=1\}.
        \]
        Now, we define elements $z_n\in B_{X\oplus_1 Y}$ by
        \[
        z_n=
        \begin{cases}
        (x_n,0), \; n\in I;\\
        (0,y_n), \; n\in J,
        \end{cases}
        \]
        where $y_n$ belongs to slice $S(B_Y, y^*_n, \alpha_n)$ for every $n\in\mathbb{N}$. Evidently, $z_n \in S_n$ for each $n\in\mathbb{N}$ and consequently, since the sequence $\{S_n\colon n\in\mathbb{N}\}$ determines the point $(a,b)$, we have that $(a,b)\in\overline{\conv}(\{ z_n\colon n\in\mathbb{N}\})$.

        Let $\varepsilon>0$. Take $\delta\in (0,1)$ such that
        $\displaystyle
        \frac{2\delta}{\norm{a}^2-\delta\norm{a}}< \varepsilon
        $
        and find a convex combination $\sum_{n=1}^{\infty}\lambda_n z_n\in \conv(\{ z_n\colon n\in\mathbb{N}\})$, where $\lambda_n\in[0,1]$ and only finitely many $\lambda_n$ are non-zero, such that
        \[
        \norm{(a,b)-\sum_{n=1}^{\infty}\lambda_nz_n}_1<\delta.
        \]
        Then,
        \begin{align*}
        \delta&>\norm{(a,b)-\sum_{n=1}^{\infty}\lambda_nz_n}_1 = \norm{(a,b)-\sum_{n\in I}\lambda_n(x_n,0) - \sum_{n\in J}\lambda_n(0,y_n)}_1\\
        &=\norm{\Big(a-\sum_{n\in I}\lambda_nx_n, b-\sum_{n\in J}\lambda_ny_n\Big)}_1 = \norm{a-\sum_{n\in I}\lambda_nx_n} + \norm{b-\sum_{n\in J}\lambda_ny_n},
        \end{align*}
        therefore
        \[
        \delta > \norm{a-\sum_{n\in I}\lambda_n x_n} \geq \norm{a}-\norm{\sum_{n\in I}\lambda_n x_n}\geq \norm{a} - \sum_{n\in I}\lambda_n,
        \]
        meaning that $\sum_{n\in I}\lambda_n> \norm{a}-\delta$. Similarly, we can see that $\sum_{n\in J}\lambda_n > \norm{b}-\delta$.

        On the other hand,
        \[
        \sum_{n\in I}\lambda_n = 1- \sum_{n\in J}\lambda_n < 1-\norm{b}+\delta = \norm{a}+\norm{b}-\norm{b}+\delta = \norm{a}+\delta,
        \]
        so in conclusion we have $\abs{\norm{a}-\sum_{n\in I}\lambda_n}< \delta$. Now, by denoting $\Lambda:= \sum_{n\in I}\lambda_n$, we have that
        \[
        \sum_{n\in I}\frac{\lambda_n}{\Lambda}x_n\in\conv(\{x_n\colon n\in\mathbb{N}\})
        \]
        and
        \begin{align*}
            \norm{\frac{a}{\norm{a}}-\sum_{n\in I}\frac{\lambda_n}{\Lambda}x_n} &= \norm{\frac{a}{\norm{a}}-\frac{\sum_{n\in I}\lambda_nx_n}{\Lambda}} = \frac{\norm{\Lambda a - \norm{a}\sum_{n\in I}\lambda_n x_n}}{\Lambda\norm{a}}\\
            &=\frac{\norm{\Lambda a - \norm{a}\sum_{n\in I}\lambda_n x_n +\norm{a}a-\norm{a}a}}{\Lambda\norm{a}}\\
            &=\frac{\norm{a(\Lambda -\norm{a}) + \norm{a}(a-\sum_{n\in I}\lambda_n x_n))}}{\Lambda \norm{a}}\\
            &\leq \frac{ \norm{a(\Lambda-\norm{a})} + \norm{\norm{a}(a-\sum_{n\in I}\lambda_n x_n)}}{\Lambda \norm{a}}\\
            &\leq \frac{\norm{a}\abs{\Lambda-\norm{a}}+ \norm{a}\norm{a-\sum_{n\in I}\lambda_nx_n}}{\Lambda \norm{a}}<\frac{\delta+ \delta}{\Lambda \norm{a}}\\
            &=\frac{2\delta}{\Lambda \norm{a}}\leq \frac{2\delta}{(\norm{a}-\delta)\norm{a}} = \frac{2\delta}{\norm{a}^2-\delta \norm{a}}<\varepsilon,
        \end{align*}
         Therefore $\frac{a}{\norm{a}}\in\overline{\conv}(\{x_n\colon n\in\mathbb{N}\})$. The proof for the point $\frac{b}{\norm{b}}$ is analogous.

        \textit{Sufficiency}. Assume $\frac{a}{\norm{a}}\in \SCD(B_X)$ and $\frac{b}{\norm{b}} \in \SCD(B_Y)$ and observe that
        \[
        (a,b) = \norm{a}\Big(\frac{a}{\norm{a}},0\Big)+\norm{b}\Big(0, \frac{b}{\norm{b}}\Big),
        \]
        where $\norm{(a,b)}_1 = \norm{a}+\norm{b} = 1$. By Proposition \ref{prop: SCD points in finite p-sum}, (a), we have
        \[
        \Big(\frac{a}{\norm{a}},0\Big)\in\SCD(B_{X\oplus_1Y})\quad \text{ and }\quad \Big(0, \frac{b}{\norm{b}}\Big)\in\SCD(B_{X\oplus_1Y}),
        \]
         and since the set of $\SCD$ points is convex, we have $(a,b)\in\SCD(B_{X\oplus_1 Y})$.
    \end{proof}

\subsection{SCD points in infinite $\ell_p$ direct sums for $1<p<\infty$}

Our aim here is to show that it is possible that $0$ is the only SCD point of a unit ball. Actually, the same result shows that it is possible to get SCD points in the unit ball of an infinite $\ell_p$-sum even if there is no SCD points of the unit ball of any of the summands. Here is the main result in this line.

\begin{theorem}\label{theorem: Y_n = c_0(I) or Daugavet}
Let $\{X_n\colon n\in \N\}$ be a sequence of Banach spaces with the Daugavet property and write $X_p:=\left[\bigoplus\nolimits_{n=1}^{\infty} X_n\right]_{\ell_p}$ for  $1<p<\infty$. Then, $\SCD(B_{X_p})=\{0\}$.
\end{theorem}

With this result we may show an example as announced as the beginning of the subsection. Consider the Banach space $X_p=\left[\bigoplus\nolimits_{n=1}^{\infty} C[0,1]\right]_{\ell_p}$, where $1<p<\infty$. Then, by Theorem \ref{theorem: Y_n = c_0(I) or Daugavet} we have $\SCD(B_{X_p})=\{0\}$. It is also worth noting that by Theorem \ref{thm: sufficient condition to have empty set of SCD points}, we know that $\SCD(B_{C[0,1]})=\emptyset$.

For the proof of Theorem~\ref{theorem: Y_n = c_0(I) or Daugavet} we will combine Propositions \ref{proposition: 0 is always SCD in infinite absolute sum} and \ref{prop: SCD points in direct sum of Daugavet + arbitrary are of the form (0,b)} below, which are interesting by themselves.

\begin{proposition}\label{proposition: 0 is always SCD in infinite absolute sum}
Let $\{X_n\colon n\in \N\}$ be a sequence of arbitrary (non-trivial) Banach spaces and write $X_p:=\left[\bigoplus\nolimits_{n=1}^{\infty} X_n\right]_{\ell_p}$ for  $1<p<\infty$. Then, $0\in \SCD(B_{X_p})$.
\end{proposition}

\begin{proof}
\parskip=0ex
    Firstly, for each $n\in\mathbb{N}$ select an element $x^*_n\in S_{X^*_n}$ and for every $k\in\mathbb{N}$ define a slice
    \[
    S_n^k = S\Big(B_X, (\underbrace{0,0,\dots,0,x_n^*}_{n \text{ components}},0,0,0,\dots), \frac{1}{k}\Big).
    \]
    Let us prove that the countable collection of slices $\{S_n^k \colon n,k\in\mathbb{N}\}$ is determining for the point $0 = (0,0,\dots)\in B_X$. To this end, pick $x^k_n\in S^k_n$ for every $n,k\in\mathbb{N}$ and fix $\varepsilon>0$. It suffices to find $x \in \overline{\conv}(\{x_n^k\colon n,k\in\mathbb{N}\})$ such that $\norm{x}< \varepsilon$.

    Pick $K\in\mathbb{N}$ such that $(p/K)^{1/p} < \varepsilon/2$ and $K^{1/p-1}<\varepsilon/2$. We will show that then $x:=\sum_{n=1}^K\frac{1}{K}x_n^K$ is the element we are looking for.

    Now for given $n\in\mathbb{N}$,
    \[
    x_n^K = (w_{n,1}^K, w_{n,2}^K, \dots, w_{n,n}^K,w_{n,n+1}^K,\dots),
    \]
    so we have
    \begin{align*}
        1 - \frac{1}{K} &< \re (0,0,\dots,0,x^*_n,0,0,\dots)(w_{n,1}^K, w_{n,2}^K, \dots, w_{n,n}^K,w_{n,n+1}^K,\dots)\\
        &=\re x^*_n(w_{n,n}^K) \leq \norm{x^*_n} \norm{w_{n,n}^K} =\norm{w_{n,n}^K},
    \end{align*}
    hence $\norm{w_{n,n}^K} > 1 - 1/K$. Moreover, we have $x_n^K \in B_X$, which means that
    \[
    \norm{x_n^K}^p=\sum_{i=1}^{\infty} \norm{w_{n,i}^K}^p\leq 1.
    \]
    Recall also the generalization of the Bernoulli inequality, which states that
    \[
    (1+s)^r > 1 + rs
    \]
    for every $r>1$ and $s\neq 0$ satisfying $s> -1$ (see \cite[p.~34]{mitrinovic1970analytic}). Now, we can derive another estimation:
    \[
    \sum_{i\neq n}^{\infty} \norm{w_{n,i}^K}^p = \Big(\sum_{i=1}^{\infty} \norm{w_{n,i}^K}^p\Big) - \norm{w_{n,n}^K}^p < 1 - \Big(1 - \frac{1}{K}\Big)^p \leq  1- \Big(1-\frac{p}{K}\Big)= \frac{p}{K}.
    \]
    Define for each $n\in\mathbb{N}$ an element
    \[
    \widehat{x}_n^K = (0,0,\dots,w_{n,n}^K,0,0,\dots).
    \]
    Using the estimation above we get
    \begin{align*}
    \norm{x^K_{n}-\widehat{x}_n^K} &= \norm{(w_{n,1}^K, \dots, w_{n,n-1}^K, 0 , w_{n,n+1}^K, \dots)} = \Big(\sum_{i\neq n}^{\infty} \norm{w_{n,i}^K}^p\Big)^{1/p}
    < \Big(\frac{p}{K}\Big)^{1/p}.
    \end{align*}
    Finally, we have that
    \begin{align*}
        \Bigg\|\sum_{n=1}^K \frac{1}{K}x_n^K\Bigg\|&\leq \Bigg\|\sum_{n=1}^K \frac{1}{K}(x_n^K-\widehat{x}_n^K)\Bigg\|+\Bigg\|\sum_{n=1}^K\frac{1}{K} \widehat{x}_n^K\Bigg\|\\
        &\leq \frac{1}{K} \sum_{n=1}^K \|x_n^K-\widehat{x}_n^K\|+\frac{1}{K}\Bigg\|\sum_{n=1}^K \widehat{x}_n^K\Bigg\|\\
        &\leq \frac{1}{K}\cdot K\cdot\Bigg(\frac{p}{K}\Bigg)^{1/p}+\frac{1}{K}\Bigg(\sum_{n=1}^K \|w_{n,n}^K\|^p\Bigg)^{1/p}\leq \varepsilon/2+\frac{1}{K}\cdot K^{1/p}<\varepsilon.\qedhere
    \end{align*}
\end{proof}

Two comments are pertinent.

\begin{remark}
{\slshape
We note that a similar result to Proposition~\ref{proposition: 0 is always SCD in infinite absolute sum} does not hold for $p=1$ or $p=\infty$.\ } Indeed, if $Z$ is a space with the Daugavet property, then so are $X:=\left[\bigoplus\nolimits_{n=1}^{\infty} Z\right]_{\ell_1}$ and $Y:=\left[\bigoplus\nolimits_{n=1}^{\infty} Z\right]_{\ell_\infty}$ by \cite[Proposition~2.16]{kadets1997Daugavet}. But $\SCD(B_X)=\SCD(B_Y)=\emptyset$ by Example~\ref{newexample:Daguavet-SCD-empty}.
\end{remark}

\begin{remark}
{\slshape
    The converse of assertion (a) of Proposition~\ref{prop: SCD points in finite p-sum} does not hold and assertion (b) of the same proposition does not hold for $a$ with $\|a\|<1$.\ } Indeed, consider the space $Z:=\left[\bigoplus_{n=1}^{\infty} C[0,1]\right]_{\ell_2}$. Then, we can write $Z=C[0,1]\oplus_2 Y$, where $Y=\left[\bigoplus_{n=2}^{\infty} C[0,1]\right]_{\ell_2}$. By Proposition~\ref{proposition: 0 is always SCD in infinite absolute sum}, $(0,0)$ is an SCD point in $B_Z$, however, $0$ is not an SCD point in $B_{C[0,1]}$ because $C[0,1]$ has the Daugavet property.
\end{remark}

We now aim to prove that in a direct sum of two spaces $E\oplus_p Y$ such that $E$ has the Daugavet property, the SCD points of the unit ball of $E\oplus_p Y$ can only be of the form $(0,b)$. To this end, we need the following result which is a consequence of \cite[Lemma~2.8]{kadets1997Daugavet}, as can be seen in the proof of \cite[Example~2.13]{aviles_slicely_2010}.

\begin{lemma}[\mbox{\cite[Lemma~2.8]{kadets1997Daugavet}, see \cite[Example~2.13]{aviles_slicely_2010}}]
\label{lemma: Daugavet x_0 not in closed linear hull}
Let X be a Banach space with the Daugavet property. Then, for every sequence of slices $\{S_n\colon n\in\mathbb{N}\}$ of $B_X$ and every $x\in S_X$, there is a sequence $\{x_n\colon n\in \mathbb N\}$ with $x_n\in S_n$ for each $n\in\mathbb{N}$, such that $x\not\in\overline{\Span}(\{x_n\colon n\in\mathbb{N}\})$.
\end{lemma}

We may now state and proof the indicated result.

\begin{proposition}
\label{prop: SCD points in direct sum of Daugavet + arbitrary are of the form (0,b)}
    Consider the Banach space $X:=E\oplus_p Y$, where $E$ has the Daugavet property, $Y$ is arbitrary, and $1<p<\infty$. If $(a,b)\in\SCD(B_X)$, then $a=0$.
\end{proposition}

\begin{proof}
\parskip=0ex
We prove the contrapositive. Let $a\neq 0$ and fix an arbitrary sequence of slices $\{S_n\colon n\in\mathbb{N}\}$, where $S_n = S(B_X, (a^*_n,b^*_n),\alpha_n)$. We may assume that $\norm{(a^*_n,b^*_n)}_q=1$ ($q$ is the H\"{o}lder conjugate of $p$). We aim to find a sequence $(x_n)$ such that $x_n\in S_n$ for each $n\in\mathbb{N}$ and $(a,b)\not\in\overline{\conv}\{x_n\colon n\in\mathbb{N}\}$. Define
        \[
        A := \{n\in\mathbb{N}\colon a^*_n\neq 0\}, \quad B= \mathbb{N}\setminus A = \{n\in\mathbb{N}\colon a^*_n= 0\}.
        \]
        Using these sets, let us define the desired sequence. First, if $n\in B$, we will pick $x_n = (0,v_n)\in S_n$, where the first coordinate can be taken zero due to the fact that $a^*_n=0$. Now, for every $n\in A$ define a slice
        \[
        T_n = S\Big(B_E,a^*_n, \frac{\alpha_n}{4}\Big),
        \]
        and consider the sequence of slices $\{T_n\colon n\in A\}$. By Lemma \ref{lemma: Daugavet x_0 not in closed linear hull}, we can find a sequence $(u_n)_{n\in A}$ so that $u_n\in T_n$ for all $n\in A$ and
        \[
        a\not\in\overline{\Span}(\{u_n\colon n\in A\}).
        \]
The latter means that there exists $\varepsilon>0$ such that $\norm{a-u}\geq \varepsilon$ for every $u\in \Span(\{u_n\colon n\in A\})$.
        We claim the following:
        \begin{equation}\label{claim: condition for pair consisting u_n}
            \forall n\in A\ \exists r_n,s_n\in \mathbb{R} \ \exists q_n\in B_{Y} \text{ such that } u_n\in T_n \ \implies \ (r_nu_n,s_nq_n)\in S_n.
        \end{equation}
Indeed, fix $n\in A$ and observe that we can find $(e_n,y_n)\in S_n$ such that
        \begin{equation*}
           \re (a^*_n,b^*_n)(e_n,y_n) =\re a^*_n(e_n)+ \re b^*_n(y_n) > 1-\frac{\alpha}{2}.
        \end{equation*}
        Take $r_n := \norm{e_n}$ and $s_n:=\norm{y_n}$. In addition, find $q_n\in B_Y$ such that
        \[
       \re  b^*_n(q_n) > \norm{b^*_n}- \frac{\alpha}{4}.
        \]
        By assumption, $u_n\in T_n$, which means that
        \[
        \re a^*_n(u_n)>\norm{a^*_n}-\frac{\alpha}{4}.
        \]
        To conclude, first notice that
        \[
        (r_nu_n,s_nq_n) = (\norm{e_n}u_n,\norm{y_n}q_n)\in B_{X}.
        \]
as we have that
        \begin{align*}
        \norm{(r_nu_n,s_nq_n)}_p^p &= \norm{(\norm{e_n}u_n,\norm{y_n}q_n)}_p^p\\
        &= \norm{e_n}^p\norm{u_n}^p + \norm{y_n}^p\norm{q_n}^p\\
        &\leq \norm{e_n}^p + \norm{y_n}^p \leq 1.
        \end{align*}
Finally, using all the estimations derived above, we get
        \begin{align*}
            \re (a^*_n,b^*_n)(r_nu_n,s_nq_n) &= \re a^*_n(u_n)\norm{e_n} + \re  b^*_n(q_n)\norm{y_n}\\
            &> \Big(\norm{a^*_n}-\frac{\alpha_n}{4}\Big)\norm{e_n}+ \Big(\norm{b^*_n}-\frac{\alpha_n}{4}\Big)\norm{y_n}\\
            &= \norm{a^*_n}\norm{e_n}+\norm{b^*_n}\norm{y_n} - \frac{\alpha_n}{4}\Big(\norm{e_n}+\norm{y_n}\Big)\\
            &\geq \re a^*_n(e_n)+\re b^*_n(y_n) -\frac{\alpha_n}{2} >1-\frac{\alpha_n}{2}- \frac{\alpha_n}{2} = 1-\alpha_n,
        \end{align*}
        which means that $(r_nu_n,s_nq_n)\in S_n$, as claimed.

Now, we are able to pick, by using the claim, for each $n\in A$
        \[
        x_n = (r_nu_n, s_nq_n)\in S_n.
        \]
        Pick arbitrary $\lambda_n\in [0,1]$, finitely many being non-zero with $\sum_{n=1}^{\infty}\lambda_n = 1$. We have
        \begin{align*}
            \norm{(a,b)-\sum_{n=1}^{\infty}\lambda_nx_n}_p &=\norm{(a,b)- \sum_{n\in A}\lambda_n(r_nu_n,s_nq_n)-\sum_{n\in B}\lambda_n (0,v_n)}_p \\
            &=\norm{\Big(a-\sum_{n\in A}\lambda_n r_nu_n, b-\sum_{n\in B}\lambda_n (s_nq_n - v_n)\Big)}_p\\
            &= \Bigg(\norm{a-\sum_{n\in A}\lambda_n r_nu_n}^p +\norm{ b-\sum_{n\in B}\lambda_n (s_nq_n - v_n)}^p\Bigg)^{1/p}\\
            &\geq \norm{a-\sum_{n\in A}\lambda_n r_nu_n}\geq \varepsilon,
        \end{align*}
        which means that $(a,b)\not\in\overline{\conv}(\{x_n\colon n\in\mathbb{N}\})$.
    \end{proof}

We can now glue the two proposition above to provide the pending proof of Theorem~\ref{theorem: Y_n = c_0(I) or Daugavet}.

\begin{proof}[Proof of Theorem~\ref{theorem: Y_n = c_0(I) or Daugavet}]
\parskip=0ex
First, we know that $0\in\SCD(B_{X_p})$ by Proposition~\ref{proposition: 0 is always SCD in infinite absolute sum}. Assume now that $(a_n)\in \SCD(B_{X_p})$. We show that $a_n=0$ for all $n\in\mathbb{N}$. Indeed, for $n\in \N$, write
        \[
        X_p=X_n\oplus_p \left[\bigoplus\nolimits_{k\neq n} X_k\right]_{\ell_p}.
        \]
Then, by Proposition~\ref{prop: SCD points in direct sum of Daugavet + arbitrary are of the form (0,b)}, we have that $a_n=0$.
    \end{proof}

\section{SCD points in projective tensor products}\label{sec: SCD points in tensor products}

Recall that the projective tensor product of two Banach spaces $X$ and $Y$, denoted by $X \widehat{\otimes}_\pi Y$, is the completion of the algebraic tensor product $X \otimes Y$ under the norm given by
$$
\|u\|:=\inf \left\{\sum_{i=1}^n\left\|x_i\right\|\left\|y_i\right\|\colon u=\sum_{i=1}^n x_i \otimes y_i\right\}.
$$
It follows easily from the definition that $B_{X \widehat{\otimes}_\pi Y}=\overline{\conv}\left(B_X \otimes B_Y\right)=\overline{\conv}\left(S_X \otimes S_Y\right)$, where
$$
A\otimes B := \{x\otimes y\colon x\in A,\; y\in B\}
$$
for subsets $A\subset X$ and $B\subset Y$. It is well known that $\left(X \widehat{\otimes}_\pi Y\right)^*=\mathcal L\left(X, Y^*\right)=\mathcal B(X\times Y)$ \cite[p.~24]{RyanTensor}.

We will now present the main result of this section.

\begin{theorem}\label{theorem: SCD tensor denting is SCD}
Let $X$ and $Y$ be Banach spaces. If $a\in\dent(B_X)$ and $b\in\SCD(B_Y)\setminus \{0\}$, then $a\otimes b\in \SCD\bigl(B_{X\pten Y}\bigr)$.
\end{theorem}

We need the following useful lemma from \cite{WernerTensor}.

\begin{lemma}[{\cite[Lemma~3.4]{WernerTensor}}]\label{lemma tensor: Werner composition}
Suppose that we have a norm one bilinear form $B\in \mathcal B(X\times Y)=(X\pten Y)^*$ and $\varepsilon>0$. Then,
\[
S(B_{X\pten Y}, B, \varepsilon^2) \subseteq \overline{\conv} (\{x\otimes y\colon x\in B_X,\, y\in B_Y,\, \real B(x,y)> 1-\varepsilon\}) + 4\varepsilon B_{X\pten Y}.
\]
\end{lemma}

\begin{proof}[Proof of Theorem~\ref{theorem: SCD tensor denting is SCD}]
\parskip=0ex
First, notice that since $a$ is a denting point of $B_X$, we can find for each $n\in\mathbb{N}$ a slice $S(B_X,x^*_n,\alpha_n)$, where $\norm{x^*_n}=1$ and
\begin{equation}\label{equation tensor: condition by using denting}
a\in S(B_X, x^*_n,\alpha_n)\subseteq B\Big(a,\frac{1}{n}\Big).
\end{equation}
On the other hand, we can find a sequence of slices
\begin{equation}\label{equation tensor: condition by using SCD}
\{S(B_Y, y^*_m,\beta_m)\colon m\in\mathbb{N},\; \norm{y^*_m}=1\},
\end{equation}
which is determining for $b$. Let us now define for each $n,m,k\in\mathbb{N}$ the following slices:
\[
S^k_{n,m} = S\Big(B_{X\pten Y}, x^*_n \otimes y^*_m, \frac{1}{k}\Big),
\]
where, as usual,
\[
(x^*_n \otimes y^*_m)(x\otimes y) = x^*_n(x) y^*_m(y)
\]
for every $x\in X$ and $y\in Y$.
Note here that the function $x^*_n\otimes y^*_m$ is bilinear and bounded, therefore $(x^*_n \otimes y^*_m)\in (X\pten Y)^*$.
Our goal is to prove that the countable collection of slices $\{S_{n,m}^k \colon n,m,k\in\mathbb{N}\}$ is determining for the elementary tensor ${a\otimes b}$. To this end, we will use Proposition~\ref{proposition: equivalent conditions of being determining for point} Condition (ii). Let $S =S (B_{X\pten Y}, B, \alpha)$, where $B$ is a norm one bounded bilinear form and  $a\otimes b\in S$. It suffices to find a member of the sequence of slices $\{S_{n,m}^k \colon n,m,k\in\mathbb{N}\}$, which is contained in $S$.

First, since $a\otimes b\in S$, we have
\[
\real B(a,b)> 1-\alpha \implies \exists \gamma>0 \ \text{ such that }\  \real B(a,b)> 1-\alpha + \gamma.
\]
This in turn means that
\[
a \in \{x\in B_X\colon \real B(x,b)>1-\alpha +\gamma\},
\]
where the set above is actually a slice of $B_X$, since it is not empty and
the mapping $x\mapsto \real B(x,b)$ is clearly linear and continuous.

Take $n\in\mathbb{N}$ such that $1/n<\gamma/32$. Then,
\begin{equation}\label{equation tensors: condition from using denting property}
a\in S(B_X,x^*_n,\alpha_n)\subseteq B\Big(a,\frac{1}{n}\Big)\subseteq B\Big(a,\frac{\gamma}{32}\Big)
\end{equation}
by (\ref{equation tensor: condition by using denting}). Using the fact that $\real B(a,b)>1-\alpha+\gamma$, we see that similarly to the last case
\[
b\in\{y\in B_Y\colon \real B(a,y)>1-\alpha +\gamma\},
\]
where the set is again a slice of $B_Y$. Since the sequence of slices in (\ref{equation tensor: condition by using SCD}) determines $b$, we can find $m\in\mathbb{N}$ such that
\begin{equation}\label{equation tensors: condition from using SCD property }
S(B_Y,y^*_m,\beta_m)\subseteq \{y\in B_Y\colon \real B(a,y)>1-\alpha +\gamma\}.
\end{equation}
Consider now the following set:
\[
S^{\otimes} := \{u\otimes v\colon u \in S(B_X,x^*_n,\alpha_n),\; v\in S(B_Y,y^*_m,\beta_m) \}.
\]
We claim that
\begin{equation}\label{equation tensors: S^tensor is contained}
S^{\otimes}\subseteq \Big\{z\in B_{X\pten Y}\colon \real B(z) >1 - \alpha + \frac{31\gamma}{32}\Big\}.
\end{equation}
Indeed, assume that $ u \in S(B_X,x^*_n,\alpha_n)$ and $v\in S(B_Y,y^*_m,\beta_m)$. By (\ref{equation tensors: condition from using SCD property }) we get $\real B(a,v)>1-\alpha+\gamma$ and using (\ref{equation tensors: condition from using denting property}), we obtain
\begin{align*}
\real B(u\otimes v) &= \real B(u,v) = \real B(a,v) - \real B(a-u,v)\\
&> 1-\alpha + \gamma - \norm{B}\norm{v}\norm{a-u}> 1-\alpha + \gamma -\frac{1}{n}\\
&> 1- \alpha + \gamma - \frac{\gamma}{32} > 1-\alpha+ \frac{31\gamma}{32}.
\end{align*}
To proceed further, pick for each $n,m\in\mathbb{N}$ another $k\in\mathbb{N}$ satisfying $1/k<\min\{\alpha_n,\beta_m\}$. Now, we claim that
\begin{equation}\label{equation tensor: S^tensor contains some set}
S\Big(B_X\otimes B_Y, x^*_n\otimes y^*_m,\frac{1}{k}\Big) \subseteq S^{\otimes}.
\end{equation}
Fix $x\otimes y \in S(B_X\otimes B_Y, x^*_n\otimes y^*_m,1/k)$, i.e.
\[
\real (x^*_n\otimes y^*_m)(x\otimes y) = \real (x^*_n(x)y^*_m(y)) > 1 - \frac{1}{k}.
\]
Pick $\theta\in \mathbb K$
with $|\theta|=1$ such that $x^*(\theta x)=\real x^*(\theta x)$;
then $x\otimes y=\theta x\otimes \theta^{-1}y$ and
\[
\real x^*(\theta x)\cdot \real (y^*_m(\theta^{-1} y))=\real (x^*_n(\theta x)y^*_m(\theta^{-1} y)) > 1 - \frac{1}{k}
\]
Since $\theta^{-1}y\in B_Y$, we have
\begin{align*}
  \real x^*(\theta x) \geq \real x^*(\theta x)\cdot \real (y^*_m(\theta^{-1} y))> 1-\frac{1}{k}>1- \alpha_n,
\end{align*}
which means that $\theta x\in S(B_X,x^*_n,\alpha_n)$. Analogously, $\theta^{-1}y\in S(B_Y,y^*_m,\beta_m)$. In conclusion
\[
x\otimes y=\theta x\otimes \theta^{-1}y \in  S^{\otimes}.
\]
Let $k\in\mathbb{N}$ be such that $4/k<\gamma/32$. In order to finish the proof,
we will show that
\begin{equation*}
S\Big(B_{X\pten Y}, x^*_n \otimes y^*_m, \frac{1}{k^2}\Big) =S_{n,m}^{k^2}\subseteq S =S(B_{X\pten Y}, B,\alpha).
\end{equation*}
Indeed, \ref{lemma tensor: Werner composition} gives that
\begin{equation}\label{equation tensor: Werner lemma}
S\Big(B_{X\pten Y}, x^*_n \otimes y^*_m, \frac{1}{k^2}\Big) \subseteq \overline{\conv}\Bigg(S\Big(B_X\otimes B_Y, x^*_n\otimes y^*_m, \frac{1}{k}\Big)\Bigg) + \frac{4}{k}B_{X\pten Y}.
\end{equation}
Pick an element $z\in S(B_{X\pten Y}, x^*_n \otimes y^*_m, 1/k^2)$.
By (\ref{equation tensor: Werner lemma}), we can write
\[
 z = a+\frac{4}{k}h \quad \text{with }\ a\in \overline{\conv}
 \Bigg(S\Big(B_X\otimes B_Y, x^*_n\otimes y^*_m, \frac{1}{k}\Big)\Bigg), \ \
 h\in B_{X\pten Y}.
\]
This means that we can find $\widehat{a}\in \conv(S(B_X\otimes B_Y, x^*_n\otimes y^*_m, 1/k))$ such that $\norm{a-\widehat{a}}<\gamma/32$.
By defining $\widehat{z} = \widehat{a}+(4/k) h$, it is obvious that
\[
\norm{z-\widehat{z}} = \norm{a-\widehat{a}} < \frac{\gamma}{32}.
\]
In addition, by (\ref{equation tensor: S^tensor contains some set}) we have
\[
\conv\bigl(S(B_X\otimes B_Y, x^*_n\otimes y^*_m, 1/k)\bigr) \subseteq \Big\{z\in B_{X\pten Y}\colon \real B(z) >1 - \alpha + \frac{31\gamma}{32}\Big\},
\]
because slices are convex. With this, we obtain
\[
\real B(\widehat{z}) = \real B(\widehat{a}) + \frac{4}{k} \real B(h) > 1-\alpha +\frac{31\gamma}{32} -\frac{4}{k}.
\]
Now, using the above estimations as well as (\ref{equation tensors: S^tensor is contained}) and (\ref{equation tensor: S^tensor contains some set}), we get
\begin{align*}
    \real B(z)&= \real B(z)- \real B(\widehat{z}) + \real B(\widehat{z}) = \real B(\widehat{z})- \real  B(\widehat{z}-z)\\
    &> 1-\alpha + \frac{31\gamma}{32} - \frac{4}{k} - \norm{B}\norm{\widehat{z}-z}
> 1-\alpha + \frac{31\gamma}{32} - \frac{4}{k} - \frac{\gamma}{32} \\
    &= 1-\alpha  + \frac{30\gamma}{32}-\frac{4}{k}> 1-\alpha + \frac{\gamma}{32} - \frac{4}{k}>1-\alpha.\qedhere
\end{align*}
\end{proof}

Let us comment that, as far as we know, it is unknown whether $X\pten Y$ is an SCD space, whenever $X$ and $Y$ are SCD spaces. The following can be considered as partial progress towards this open question.

\begin{corollary}\label{corollary tensor: SCD and dentable unit balls give SCD unit ball}
Let $X$ and $Y$ be Banach spaces such that
$B_X=\overline{\conv}(\dent{B_X})$ and
$\SCD(B_Y)=B_Y$. Then, $\SCD(B_{X\pten Y}) = B_{X\pten Y}$. If additionally $X$ and $Y$ are separable, then $B_{X\pten Y}$ is an SCD set.
\end{corollary}

\begin{proof}
\parskip=0ex
    Notice that using elementary properties of projective tensor products together with the assumption, we can write
    \begin{align*}
        B_{X\pten Y} &= \overline{\conv}(B_X\otimes B_Y) = \overline{\conv}\big(\overline{\conv}(\dent(B_X))\otimes \SCD(B_Y)\big)\\
        &= \overline{\conv}(\dent(B_X)\otimes \SCD(B_Y)).
    \end{align*}
    By Theorem \ref{theorem: SCD tensor denting is SCD},
    \begin{equation}\label{eq tensor for corollary: implication of main theorem}
      \dent(B_X)\otimes (\SCD(B_Y)\setminus\{0\}) \subseteq \SCD(B_ {X\pten Y}).
    \end{equation}
    This in turn implies that
    \begin{equation}\label{eq tensor for corollary: implication of main theorem improvement}
    \dent(B_X)\otimes \SCD(B_Y)\subseteq \SCD(B_{X\pten Y}).
    \end{equation}
    Indeed, pick $x\otimes y \in \dent(B_X)\otimes \SCD(B_Y)$, where $y\neq 0$
    (such elements clearly do exist).
    Then by (\ref{eq tensor for corollary: implication of main theorem})
    we directly obtain that $x\otimes y\in \SCD(B_ {X\pten Y})$.
    Using Corollary \ref{remark: 0 being SCD}, we get $0\in\SCD(B_{X\pten Y})$.
    Therefore, picking $x\otimes 0 \in \dent(B_X)\otimes \SCD(B_Y)$, we have
    \[
    x\otimes 0 = 0\in\SCD(B_{X\pten Y})
    \]
    and consequently (\ref{eq tensor for corollary: implication of main theorem improvement}) holds.

    As the set of SCD points of the unit ball is closed and convex by
    Lemma~\ref{lemma: properties of SCD(A)}, we obtain
    \[
    \overline{\conv}(\dent(B_X)\otimes \SCD(B_Y)) \subseteq \SCD(B_{X\pten Y}).
    \]
    In conclusion
    \[
    B_{X\pten Y} =\overline{\conv}(\dent(B_X)\otimes \SCD(B_Y)) \subseteq \SCD(B_{X\pten Y}) \subseteq B_{X\pten Y}.
    \]
    If additionally $X$ and $X$ are separable, so is the set $B_{X\pten Y}$, hence by Lemma~\ref{lemma: connection between SCD sets and SCD points} we have that $B_{X\pten Y}$ is an SCD set.
\end{proof}

The next result gives another sufficient condition for the elementary tensor of two SCD points to be also an SCD point.

\begin{theorem}\label{theorem: a SCD b preserved extreme then a tensor b SCD}
Let $X$ and $Y$ be Banach spaces, and let $a\in \SCD(B_X)$ and $b\in \SCD(B_Y)$. Assume that
\begin{enumerate}
    \item every operator $T\colon X\rightarrow Y^*$ is compact;
    \item $b$ is a preserved extreme point.
\end{enumerate}
Then, $a\otimes b\in \SCD(B_{X\pten Y})$.
\end{theorem}

\begin{proof}
\parskip=0ex
Let $S_n=S(B_X, x^*_n,\alpha_n)$ and $T_m=S(B_Y, y^*_m,\beta_m)$, where $\|x^*_n\|=\|y^*_m\|=1$ and $\alpha_n,\beta_m>0$, be the determining sequences of slices for $a$ and $b$, respectively.

 Let us now define for each $n,m,k\in\mathbb{N}$ the following slices:
\[
S^k_{n,m} = S\Big(B_{X\pten Y}, x^*_n \otimes y^*_m, \frac{1}{k}\Big),
\]
where
\[
(x^*_n \otimes y^*_m)(x\otimes y) = x^*_n(x) y^*_m(y)
\]
for every $x\in X$ and $y\in Y$.
Note here that the function $x^*_n\otimes y^*_m$ is bilinear and bounded, therefore $(x^*_n \otimes y^*_m)\in (X\pten Y)^*$.
Our goal is to prove that the countable collection of slices $\{S_{n,m}^k \colon n,m,k\in\mathbb{N}\}$ is determining for the elementary tensor $a\otimes b$.

Let $S=S(B_{X\pten Y}, T, \alpha)$ be a slice containing the element $a\otimes b$, where $T\in \mathcal{L}(X,Y^*)$ and $\alpha>0$. Find $\eta>0$ such that
\[
\real T(a)(b)>1-\alpha+\eta.
\]
Since $a\in \{x\in B_X\colon \real T(x)(b)>1-\alpha+\eta\}$ and $a\in\SCD(B_X)$, then there exists $n\in \mathbb N$ such that
\[
S(B_X, x^*_n,\alpha_n)\subseteq \{x\in B_X\colon \real T(x)(b)>1-\alpha+\eta\}.
\]
By our assumption $T$ is compact, so there are $x_1,\dots, x_p\in S(B_X, x^*_n,\alpha_n)$ such that
\[
T(S(B_X, x^*_n,\alpha_n))\subseteq \bigcup_{i=1}^p B\left(T(x_i), \frac{\eta}{16}\right).
\]
Since, $x_i\in \{x\in B_X\colon \real T(x)(b)>1-\alpha+\eta\}$ for every $1\leq i\leq p$, then
\[
b\in \bigcap_{i=1}^p\{y\in B_Y\colon \real T(x_i)(y)>1-\alpha+\eta\}=:W.
\]
Note that $W$ is a relatively weakly open subset of $B_Y$ containing $b$, which is a preserved extreme point, hence we can find a slice $S$ of $B_Y$ such that $b\in S\subseteq W$ \cite{LLT1988}. Since $b\in \SCD(B_Y)$, we can find $m\in \mathbb N$ such that
\[
S(B_Y, y^*_m,\beta_m)\subseteq S\subseteq W.
\]
Consider now the following set
\[
S^{\otimes} := \{u\otimes v\colon u \in S(B_X,x^*_n,\alpha_n),\; v\in S(B_Y,y^*_m,\beta_m) \}.
\]
Then, by arguing as in the proof of Theorem~\ref{theorem: SCD tensor denting is SCD}, one has that
\begin{equation}\label{equation:last-theorem4.4}
S^{\otimes}\subseteq \Big\{z\in B_{X\pten Y}\colon \real B(z) >1 - \alpha + \frac{15\eta}{16}\Big\}.
\end{equation}
Finally, the proof follows by repeating the arguments in the last part of the proof of Theorem~\ref{theorem: SCD tensor denting is SCD}, using now \eqref{equation:last-theorem4.4} instead of \eqref{equation tensors: condition from using SCD property }.
\end{proof}

\section{SCD points in Lipschitz-free spaces}\label{sec: SCD points in Lipschitz-free spaces}

Let $M$ be a metric space with metric $d$ and a fixed point 0. We denote by $\operatorname{Lip}_0(M)$ the \emph{real} Banach space of all Lipschitz functions $f\colon M \rightarrow \mathbb{R}$ with $f(0)=0$ equipped with the norm
$$
\|f\|:=\sup \left\{\frac{|f(x)-f(y)|}{d(x, y)}\colon x, y \in M,\ x \neq y\right\} .
$$
Let $\delta\colon M \rightarrow \operatorname{Lip}_0(M)^*$ be the canonical isometric embedding of $M$ into $\operatorname{Lip}_0(M)^*$ given by $x \mapsto \delta_x$, where $\delta_x(f)=f(x)$. The norm closed linear span of $\delta(M)$ in $\operatorname{Lip}_0(M)^*$ is called the \emph{Lipschitz-free space} over $M$ and is denoted by $\mathcal{F}(M)$ (see \cite{godefroySurvey2015} and \cite{weaver2018} for the background). It is well known that $\mathcal{F}(M)^*=\operatorname{Lip}_0(M)$. An element in $\mathcal{F}(M)$ of the form
\[
m_{x,y}:=\frac{\delta_x-\delta_y}{d(x,y)}
\]
for $x,y\in M$ with $x\neq y$ is called a \emph{molecule}.

The main result of the section is a characterization of SCD point of the unit ball of a Lipschitz-free space for \emph{complete} metric spaces.

\begin{theorem}\label{theo:freecomplete}
Let $M$ be a complete metric space and let $\mu\in S_{\mathcal F(M)}$. Then, the following are equivalent:
\begin{enumerate}
    \item[\rm{(i)}] $\mu\in\SCD(B_{\mathcal F(M)})$;
    \item[\rm{(ii)}] $\mu\in \overline{\conv}\bigl(\dent(B_{\mathcal F(M)})\bigr)$.
\end{enumerate}
In particular, if $M$ is separable, then $B_{\mathcal F(M)}$ is an SCD set if, and only if, $B_{\mathcal F(M)}=\overline{\conv}\bigl(\dent(B_{\mathcal F(M)})\bigr)$.
\end{theorem}

We need a preliminary result which is interesting by itself. Recall that a Lipschitz function $f\in \Lip_0(M)$ is said to be \emph{local} if, for every $\varepsilon>0$, there exists $u,v\in M$ with $0<d(u,v)<\varepsilon$ and $\frac{f(u)-f(v)}{d(u,v)}>\Vert f\Vert-\varepsilon$. In short, a local Lipschitz function is a function whose Lipschitz norm can be approximated by pairs of points which are arbitrarily close. The next lemma shows that only non-local functions are needed to construct a sequence of determining slices for a norm-one element of the unit ball of a Lipschitz-free space.

\begin{lemma}\label{lemma:freenonlocalsequence}
Let $M$ be a metric space and $\mu\in \SCD(B_{\mathcal F(M)})\cap S_{\mathcal{F}(M)}$. Assume that $S_n=S(B_{\mathcal F(M)},f_n,\alpha_n)$ is a determining sequence for $\mu$. Set
$$I:=\{n\in\mathbb N\colon f_n\mbox{ is not local }\}.$$
Then, $\{S_n\colon n\in I\}$ is determining for $\mu$.
\end{lemma}

\begin{proof}
\parskip=0ex
Assume, by contradiction, that $\{S_n\colon n\in I\}$ is not determining for $\mu$. Consequently, for every $n\in I$ there exist $x_n\in S_n$ satisfying that $\mu\notin\overline{\conv}(\{x_n\colon n\in I\})$. By the Hahn-Banach theorem, there exists $f\in S_{\Lip_0(M)}$ and $\alpha>0$ such that
$$
f(\mu)>\alpha>\sup\bigl\{f(z)\colon z\in \overline{\conv}\{x_n\colon  n\in I\}\bigr\}.
$$
Furthermore, we can  find $0<\beta<\alpha$ satisfying
$$
f(\mu)>\alpha>\beta>f(x_n)\ \  \forall n\in I.
$$
Select $\varepsilon, \eta>0$ small enough so that
\begin{equation}\label{eq:epsetafree}
(1-\beta)(\eta+2\varepsilon) + \eta <\alpha-\beta.
\end{equation}
Now, set $J=\mathbb N\setminus I=\{n\in\mathbb N\colon  f_n\mbox{ is local }\}$ and write $J=\{k_n\colon n\in\mathbb N\}$ (admitting that $k_n$ may be eventually constant if $J$ is finite). Choose a sequence $(\varepsilon_n)$ of positive real numbers such that $1-\varepsilon<\prod_{n=1}^\infty (1-\varepsilon_n)$. Our aim is to construct, by induction, a sequence $x_{k_n}\in S_{k_n}$ with the property that
\begin{equation}\label{eq:condixknnolocal}
\left\Vert\mu+\sum_{i=1}^n \lambda_i x_{k_i} \right\Vert>\prod_{i=1}^n (1-\varepsilon_n)\Big(1+\sum_{i=1}^n \vert\lambda_i\vert\Big)
\end{equation}
for every $n\in\mathbb N$ and $\lambda_1,\ldots, \lambda_n\in\mathbb R$. Let us construct $x_{k_1}$. Since $f_{k_1}$ is local, we can find a sequence of points $u_n, v_n\in M$ with $0<d(u_j,v_j)\rightarrow 0$ such that $f_{k_1}(m_{u_j,v_j})=\frac{f_{k_1}(u_j)-f_{k_1}(v_j)}{d(u_j,v_j)}>1-\alpha_{k_1}$ or, in other words, that $m_{u_j,v_j}\in S_{k_1}$ holds for every $j\in\mathbb N$. Since $d(u_j,v_j)\rightarrow 0$, \cite[Theorem 2.6]{jung2021daugavet} implies that
$$\Vert \nu +m_{u_j,v_j}\Vert\rightarrow 1+\Vert \nu\Vert$$
holds for every $\nu\in \mathcal F(M)$. Consequently, we can find $j\in\mathbb{N}$ big enough so that $x_{k_1}:=m_{u_j,v_j}$ satisfies
$$
\Vert \mu\pm x_{k_1}\Vert>2-\frac{\varepsilon_1}{2}.
$$
Notice that an application of \cite[Lemma~2.2]{kadets2020lemma} secures us that
\[
\Vert \mu +\lambda x_{k_1}\Vert>(1-\varepsilon_1)( 1+\vert \lambda \vert)
\]
holds for every $\lambda\in\mathbb R$, hence equation (\ref{eq:condixknnolocal}) for $k_1$ is satisfied.


Now assume the inductive step that $x_{k_1},\ldots x_{k_n}$ has been constructed with the desired property, and let us construct $x_{k_{n+1}}$. In order to do so, let $Y:=\Span\{\mu, x_{k_1},\ldots, x_{k_n}\}$, which is a finite-dimensional subspace of $\mathcal F(M)$. Since $S_Y$ is compact as $Y$ is finite-dimensional, we can select a finite set $F\subseteq S_Y$ which is an $\frac{\varepsilon_{n+1}}{2}$-net for $S_Y$. Once again the condition that $f_{k_{n+1}}$ is local allows us to guarantee the existence of a sequence $m_{u_j,v_j}\in S_{k_{n+1}}$ such that $d(u_j,v_j)\rightarrow 0$. Since
$$
\Vert \nu+m_{u_j,v_j}\Vert\rightarrow 2
$$
for every $\nu\in F$, again by \cite[Theorem 2.6]{jung2021daugavet}, we can find $j\in\mathbb{N}$ large enough so that, if we select $x_{k_{n+1}}:=m_{u_j,v_j}$, we have $\Vert \nu+x_{k_{n+1}}\Vert>2-\frac{\varepsilon_{n+1}}{2}$ for every $\nu\in F$ (since $F$ is finite). As $F$ is an $\varepsilon_{n+1}/2$-net, it follows that the inequality
$$
\Vert \nu+x_{k_{n+1}}\Vert>2-\varepsilon_{n+1}
$$
holds for every $\nu\in S_Y$. From here it can be proved that
$$\Vert \nu+\lambda x_{k_{n+1}}\Vert>(1-\varepsilon_{n+1})(\Vert \nu\Vert+\vert \lambda\vert)$$
holds for every $\nu\in Y$ and every $\lambda\in\mathbb R$. Let us prove that $x_{k_1},\ldots, x_{k_{n+1}}$ satisfy the desired condition. In order to do so, select $\lambda_1,\ldots, \lambda_{n+1}\in\mathbb R$. Observe that, since $\mu+\sum_{i=1}^n \lambda_i x_{k_i}\in Y$, we obtain
$$
\left\Vert \mu+\sum_{i=1}^{n+1} \lambda_i x_{k_i}\right\Vert\geq (1-\varepsilon_{n+1})\left(\left\Vert \mu+\sum_{i=1}^n \lambda_i x_{k_i} \right\Vert+\vert \lambda_{n+1}\vert \right).
$$
Now, the inductive step implies $\Vert \mu+\sum_{i=1}^n \lambda_i x_{k_i} \Vert\geq \prod_{i=1}^n (1-\varepsilon_i)(1+\sum_{i=1}^n\vert \lambda_i\vert)$, so
\begin{align*}
\left\Vert \mu+\sum_{i=1}^{n+1} \lambda_i x_{k_i}\right\Vert& \geq (1-\varepsilon_{n+1})\left(\prod_{i=1}^n (1-\varepsilon_i)\left(1+\sum_{i=1}^n\vert \lambda_i\vert\right ) +\vert \lambda_{n+1}\vert\right)\\
& \geq (1-\varepsilon_{n+1})\left(\prod_{i=1}^n (1-\varepsilon_i)\left(1+\sum_{i=1}^n\vert \lambda_i\vert\right ) +\prod_{i=1}^n (1-\varepsilon_i)\vert \lambda_{n+1}\vert\right)\\
& =\prod\limits_{i=1}^{n+1}(1-\varepsilon_i)\left(1+\sum_{i=1}^{n+1} \vert \lambda_i\vert \right)
\end{align*}
which finishes the proof of the construction of $x_{k_n}$.

As we have $x_n\in S_n$ for every $n\in\mathbb N$ and $\{S_n\colon  n\in\mathbb N\}$ is determining for $\mu$, we conclude $\mu\in \overline{\conv}(\{x_n\colon n\in\mathbb N\})$. Consequently, we can find $(\lambda_n)\subseteq [0,1]$ with $\sum_{n=1}^\infty \lambda_n=1$ and only finitely many $\lambda_n$ being non-zero, satisfying
\begin{equation}\label{eq:freeapprconv}
\left\Vert \mu-\sum_{n=1}^\infty \lambda_n x_n\right\Vert<\eta.
\end{equation}
If we evaluate at $f$, we obtain
\begin{align*}
\eta>f\left(\mu-\sum_{n=1}^\infty \lambda_n x_n\right) &=f(\mu)-\sum_{n=1}^\infty \lambda_n f(x_n) >\alpha-\sum_{n\in I}\lambda_n f(x_n)-\sum_{n\in J}\lambda_n f(x_n)\\
& \geq \alpha-\beta\sum_{n\in I}\lambda_n -\sum_{n\in J}\lambda_n =\alpha-\beta\left(1-\sum_{n\in J}\lambda_n\right)-\sum_{n\in J}\lambda_n \\ &=\alpha-\beta-(1-\beta)\sum_{n\in J}\lambda_n,
\end{align*}
hence
\begin{equation}\label{eq:sumJmass}
\sum_{n\in J}\lambda_n>\frac{\alpha-\beta-\eta}{1-\beta}.
\end{equation}
On the other hand, by the construction of $x_n, n\in J$, we have from (\ref{eq:condixknnolocal}) that
$$\left\Vert \mu-\sum_{n\in J}\lambda_n x_n \right\Vert> (1-\varepsilon)\Big(1+\sum_{n\in J}\lambda_n\Big).$$
These estimations imply that
\begin{align*}
    \eta &> \norm{\mu-\sum_{n=1}^{\infty}\lambda_nx_n}\geq \norm{\mu-\sum_{n\in J}\lambda_nx_n} - \norm{\sum_{n\in I}\lambda_nx_n}\\
    &>(1-\varepsilon)\Big(1+\sum_{n\in J}\lambda_n\Big) - \sum_{n\in I}\lambda_n = (1-\varepsilon)\Big(1+\sum_{n\in J}\lambda_n\Big) - \Big(1-\sum_{n\in J}\lambda_n\Big) \\
    &=2\sum_{n\in J}\lambda_n - \varepsilon\Big(1+\sum_{n\in J}\lambda_n\Big)\geq 2\sum_{n\in J}\lambda_n - 2\varepsilon > 2\frac{\alpha-\beta-\eta}{1-\beta}-2\varepsilon\\
    &>\frac{\alpha-\beta-\eta}{1-\beta}-2\varepsilon.
    \end{align*}
    This, in turn, claims that $\alpha-\beta < (1-\beta)(\eta+2\varepsilon)+\eta$, which disagrees with (\ref{eq:epsetafree}). This contradiction finishes the proof.
\end{proof}

We are now ready to present the pending proof.

\begin{proof}[Proof of Theorem~\ref{theo:freecomplete}]
\parskip=0ex
The proof of (ii)$\Rightarrow$(i) is immediate since every denting point is SCD (Example~\ref{example:dentingpoints}) and  the set of SCD points of a set is closed and convex (Lemma~\ref{lemma: properties of SCD(A)}).

For the proof of (i)$\Rightarrow$(ii), assume that $\mu\in\SCD(B_{\mathcal F(M)})$ and take a determining sequences of slices $S_n=S(B_{\mathcal F(M)}, f_n,\alpha_n)$. By Lemma \ref{lemma:freenonlocalsequence} we can assume with no loss of generality that $f_n\in S_{\Lip_0(M)}$ is non-local for every $n\in\mathbb N$. Given $n\in\mathbb N$, since $M$ is complete and $f_n$ is non-local, \cite[Proposition~2.7]{veeorg2022characterizations} implies that there exists a denting point $x_n$ of $B_{\mathcal F(M)}$ with $x_n\in S_n$ for every $n\in\mathbb N$. Since $\{S_n\colon n\in \N\}$ is determining, we get that $\mu\in \overline{\conv}(\{x_n\colon n\in\mathbb N\})\subseteq \overline{\conv}(\dent(B_{\mathcal F(M)})$, as requested.

The final remark follows since $\mathcal F(M)$ is separable because $M$ is separable.
\end{proof}

By additionally requiring the underlying metric space of the Lipschitz-free space to be compact and taking into account the proof of Theorem \ref{theo:freecomplete}, we can prove the following.

\begin{theorem}\label{theo:freecompact}
Let $M$ be a compact space and let $\mu\in S_{\mathcal F(M)}$. The following are equivalent:
\begin{enumerate}
    \item[\rm{(i)}] $\mu\in\SCD(B_{\mathcal F(M)})$;
    \item[\rm{(ii)}] $\mu\in \overline{\conv}\bigl(\strexp(B_{\mathcal F(M)})\bigr)$.
\end{enumerate}
In particular, $B_{\mathcal F(M)}$ is an SCD set if, and only if, $B_{\mathcal F(M)}=\overline{\conv}\bigl(\strexp(B_{\mathcal F(M)})\bigr)$.
\end{theorem}

\begin{proof}
\parskip=0ex
The proof of (ii)$\Rightarrow$(i) is immediate since every strongly exposed point is denting and because the set of SCD points is closed and convex.

For the proof of (i)$\Rightarrow$(ii), assume that $\mu\in\SCD(B_{\mathcal F(M)})$  and take a determining sequence of slices $S_n=S(B_{\mathcal F(M)}, f_n,\alpha_n)$. By Lemma \ref{lemma:freenonlocalsequence} we can assume with no loss of generality that $f_n\in S_{\Lip_0(M)}$ is non-local for every $n\in\mathbb N$. Given $n\in\mathbb N$, since $M$ is compact and $f_n$ is non-local, \cite[Lemma 3.13]{cgmr21} implies that $f_n$ attains its norm at a strongly exposed point $x_n$ of $B_{\mathcal F(M)}$ which clearly belongs to $S_n$.  Since $\{S_n\colon n\in \N\}$ is determining for $\mu$, we get that $\mu\in \overline{\conv}(\{x_n\colon  n\in\mathbb N\})\subseteq \overline{\conv}\bigl(\strexp(B_{\mathcal F(M)})\bigr)$, as requested.

The last statement follows since $\mathcal F(M)$ is separable, because $M$ is separable as being a compact metric space.
\end{proof}

\section{Applications}\label{sec: applications}

In this section, we aim to establish some connections between SCD points and some related properties.

We begin by studying the inheritance of Daugavet property to its subspaces. It is known that if $X$ is a Banach space with the Daugavet property and $Y$ is a subspace of $X$, then $Y$ has the Daugavet property whenever $Y$ is an $M$-ideal in $X$ \cite[Proposition~2.10]{kadets1997Daugavet} or $(X/Y)^*$ is separable \cite[Theorem~2.14]{kadets1997Daugavet}. We now enlarge this list by proving that the subspace $Y$ also inherits the Daugavet property whenever $0\in B_{X/Y}$ is an SCD point in every convex set $C\subseteq B_{X/Y}$ containing it.

\begin{theorem}\label{theo:inheridaugascd}
Let $X$ be a Banach space and let $Y$ be a subspace of $X$. Assume that $X$ has the Daugavet property and that $0\in X/Y$ satisfies that $0$ is an SCD point in any convex subset $C\subseteq B_{X/Y}$ containing it. Then, $Y$ has the Daugavet property.
\end{theorem}

\begin{proof}
\parskip=0ex
Let $y_0\in S_Y$, $\varepsilon>0$ and let $S=S(B_Y,y^*,\alpha)$ be a slice of $B_Y$ with $y^*\in S_{X^*}$. Let us find $y\in S$ such that $\Vert y_0-y\Vert>2-\varepsilon$. In such a case, we will have that $Y$ has the Daugavet property by \cite[Lemma~2.2]{kadets1997Daugavet}. In order to do so, write $T=S(B_X,y^*,\beta)$ for some $0<\beta<\alpha$. Let $p\colon X\longrightarrow X/Y$ be the quotient mapping and notice that, since $T\cap S_Y\neq \emptyset$, we can guarantee that $0$ belongs to the convex subset $p(T)$ of $B_{X/Y}$. By the assumption, $0\in \SCD(p(T))$, so we can find a determining sequence of slices $S_n\subseteq p(T)$ for $0$. Now, observe that the sets
$$
W_n:=T\cap p^{-1}(S_n)\qquad (n\in \N)
$$
are non-empty, weakly open, and contained in $B_X$.

Since $X$ has the Daugavet property, then for any $\delta>0$, there exists a sequence $(x_n)$ with $x_n\in W_n$ for every $n\in \N$, satisfying that
\begin{equation}\label{eq:ineritdauga}\left\Vert y_0-\sum_{n=1}^\infty \lambda_n x_n \right\Vert>2-\delta
\end{equation}
for every $(\lambda_n)\in S_{\ell_1}$. Indeed, this can be seen using the analogue of \cite[Lemma~2.8]{kadets1997Daugavet} for relative weakly open set which holds applying in the proof \cite[Lemma~3]{shvidkoy2000} instead of \cite[Lemma~2.1~(a)]{kadets1997Daugavet}.

Since $x_n\in W_n=T\cap p^{-1}(S_n)$, we clearly have that $p(x_n)\in S_n$. Taking into account that the sequence $\{S_n\colon n\in\mathbb{N}\}$ is determining for $0\in p(T)$, there exists $(\lambda_n)\in S_{\ell_1}$ such that
$$
\left\Vert \sum_{n=1}^\infty \lambda_n p(x_n) \right\Vert_{X/Y}<\delta.
$$
This means that there exists $z\in Y$ such that
$$\left\Vert z-\sum_{n=1}^\infty \lambda_n x_n\right\Vert<\delta.$$
Observe that $\Vert z\Vert\leq 1+\delta$. Furthermore, \eqref{eq:ineritdauga} implies that $\Vert \sum_{n=1}^\infty \lambda_n x_n\Vert>1-\delta$ and thus $\Vert z\Vert\geq 1-2\delta$. Finally set $y=\frac{z}{\Vert z\Vert}$. It is clear that $\Vert y-z\Vert\leq 2\delta$ so $\Vert y-\sum_{n=1}^\infty \lambda_n x_n\Vert\leq 3\delta$ by the triangle inequality. Let us prove that $y$ fits our requirements. Firstly, since $x_n\in W_n\subseteq T$ and $T$ is convex, we infer $\sum_{n=1}^\infty \lambda_n x_n\in T$, so $\re y^*(\sum_{n=1}^\infty \lambda_n x_n)>1-\beta$. Hence $\re y^*(y)>1-\beta-3\delta$. On the other hand, \eqref{eq:ineritdauga} implies
$$\Vert y_0-y\Vert>2-4\delta.$$
Consequently, $y\in S(B_Y,y^*,\alpha)$ and $\Vert y_0-y\Vert>2-\varepsilon$ as soon as $4\delta<\varepsilon$ and $\beta+3\delta<\alpha$, which finishes the proof.
\end{proof}

The following particular case is specially interesting and extends some of the previously known cases exposed before. If follows directly from Theorem~\ref{theo:inheridaugascd} by using Example~\ref{example:stronglyregular-set}.

\begin{corollary}
Let $X$ be a Banach space with the Daugavet property and let $Y$ be a subspace of $X$. If $X/Y$ is strongly regular (in particular, CPCP or RNP), then $Y$ has the Daugavet property.
\end{corollary}

Our next aim is to show that separable Banach spaces for which every convex series of slices of $B_X$ intersects the sphere, must contain an isomorphic copy of $\ell_1$.

\begin{theorem}\label{thm: separable X contains ell1 whenever X* fails (-1)-BCP}
Let $X$ be a Banach space such that every convex series of slices of $B_X$ intersects the unit sphere. Then, $\SCD(B_X)=\emptyset$. If, moreover, $X$ is separable, then $X$ is not an SCD space and, in particular, it contains an isomorphic copy of $\ell_1$.
\end{theorem}

\begin{proof}
\parskip=0ex
Assume that every convex series of slices of $B_X$ intersects the sphere. This means that for every sequence of slices $\{S_n\colon n\in\mathbb{N}\}$ of $B_X$ there are $x_n\in S_n\cap S_X$ and $x^*\in S_{X^*}$ such that $x^*(x_n)=1$ for every $n\in \mathbb{N}$ \cite[Lemma~2.8]{CIACI2022126185}. In particular, $0\notin \overline{\conv}(\{x_n\colon n\in \mathbb{N}\})$. This implies that $0\notin \SCD(B_X)$, hence $\SCD(B_X)=\emptyset$ by Corollary~\ref{remark: 0 being SCD}. If moreover $X$ is separable, it cannot be an SCD space, so $X$ contains a $\ell_1$-sequence (see \cite[Theorem 2.22]{aviles_slicely_2010}).
\end{proof}

Some comments on the above result are pertinent.

\begin{remark}
We collect here some comments on Theorem~\ref{thm: separable X contains ell1 whenever X* fails (-1)-BCP}:
\begin{enumerate}
\item {\slshape It also follows from the above theorem that if every convex series of slices of $B_X$ intersects the unit sphere, then $B_X$ does not contain strongly regular points (by using Proposition~\ref{proposition: SCS is SCD}), but this follows directly from the hypothesis and it does not need Theorem~\ref{thm: separable X contains ell1 whenever X* fails (-1)-BCP}.}

\item {\slshape The assumption on convex series of slices intersecting the sphere cannot be relaxed to (finite) convex combinations of slices intersecting the unit sphere.\ } Indeed, the separable Banach space $c_0$ has the property that every finite convex combination of slices of $B_X$ intersects the unit sphere \cite{Abrahamsen_Lima_2018}, but it does not contain $\ell_1$.

\item {\slshape Note that the second part of Theorem~\ref{thm: separable X contains ell1 whenever X* fails (-1)-BCP} cannot be extended to non-separable Banach spaces.\ } Indeed, if $I$ is uncountable, then every convex series of slices of $B_{c_0(I)}$ intersects the unit sphere \cite[Proposition~3.6 and Example~6.2]{CIACI2022126185}, but $c_0(I)$ does not contain $\ell_1$.
\end{enumerate}
\end{remark}

Our final application shows the validity of the Daugavet equation \eqref{eq:DE} (i.e.\ $\|\Id+T\|=1+\|T\|$) for those operators $T$ on a Banach space with the Daugavet property for which one can almost reach the norm or $T$ using SCD points of $T(B_X)$. It is actually the pointwise version of \cite[Proposition~5.8]{aviles_slicely_2010}.

\begin{theorem}\label{theorem: SCD points guarantee DE}
Let $X$ be a Banach space with the Daugavet property and $T\in \mathcal{L}(X,X)$. Suppose that for every $\varepsilon>0$ there exists an element $Tx\in \SCD\bigl(T(B_X)\bigr)$ such that $\norm{Tx}> \|T\|- \varepsilon$. Then, $T$ satisfies \eqref{eq:DE}.
    \end{theorem}

    \begin{proof}
\parskip=0ex
        Let us first introduce some notation. We denote $K(X^*)$ as the intersection of $S_{X^*}$ with the weak$^*$-closure in $X^*$ of $\ext(B_{X^*})$, which is a (James) boundary for $X$ (as it contains $\ext(B_{X^*})$). Secondly, for a slice $S$ of $B_X$ and $\varepsilon>0$, write
        \[
        D(S,\varepsilon) = \{y^*\in K(X^*)\colon S\cap \overline{\conv}\big(S(B_X,y^*,\varepsilon)\big)\neq\emptyset\}.
        \]
        Assume now, without loss of generality, that $\norm{T}=1$. By assumption, we can pick $x\in B_X$ such that $Tx$ is an SCD point of $T(B_X)$ and
        \[
       \norm{Tx}>1- \frac{\varepsilon}{2}.
        \]
       Let $\{S_n\colon n\in\mathbb{N}\}$ be a determining sequence of slices of $T(B_X)$ for $Tx$, and observe that $T^{-1}(S_n)\cap B_X$ are slices of $B_X$.
       Since $K(X^*)$ is a boundary for $X$ and it is balanced, one can find $y^*_0\in K(X^*)$ so that $\re y^*_0(Tx)=\norm{Tx}$, which in particular, implies that $\re y^*_0(Tx)>1-\varepsilon/2$.
       By \cite[Proposition~4.3~(v)]{aviles_slicely_2010}, the set $\bigcap_{n\in\mathbb{N}}D(T^{-1}(S_n),\varepsilon/2)$ is weak$^{*}$-dense in $K(X^*)$. Therefore, we are able to find an element $y^*\in\bigcap_{n\in\mathbb{N}}D(T^{-1}(S_n),\varepsilon/2)$ satisfying
         \[
         \bigl|\real(y^*-y^*_0)(Tx)\bigr|<\frac{\varepsilon}{2}.
         \]
         Utilizing the estimations derived above, we obtain
         \[
         \real y^*(Tx) =\real y^*_0(Tx) + \real (y^*-y^*_0)(Tx) > 1-\frac{\varepsilon}{2} - \frac{\varepsilon}{2} = 1-\varepsilon,
         \]
         hence $Tx\in S(B_X,y^*,\varepsilon)$.

         We wish now to use the fact that $Tx$ is an SCD point for $T(B_X)$. In order to do so, observe that for each $n\in\mathbb{N}$
         \[
    \overline{\conv}\big(S(B_X,y^*,\varepsilon)\big)\cap T^{-1}(S_n)\neq \emptyset
         \]
         by the definition of the set $D(T^{-1}(S_n),\varepsilon/2)$. This means that
         \[ T\Big(\overline{\conv}\big(S(B_X,y^*,\varepsilon)\big)\Big)\cap S_n\neq \emptyset
         \]
         for every $n\in\mathbb{N}$, and since $\{S_n\colon n\in\mathbb{N}\}$ is determining for $Tx$, we infer that
         \[
         Tx\in \overline{\conv}\Big(T\Big(\overline{\conv}\big(S(B_X,y^*,\varepsilon)\big)\Big)\Big) = \overline{T\Big(\conv\big(S(B_X,y^*,\varepsilon)\big)\Big)}.
         \]
         This enables us to find
         \[
         z\in T\Big(\conv\big(S(B_X,y^*,\varepsilon)\big)\Big)\text{ such that } \norm{Tx-z}<\varepsilon,
         \]
         which, in particular, implies that
         \[
         \varepsilon>\real y^*(Tx-z) = \real y^*(Tx)-\real y^*(z)
         \]
         and, therefore,
         \begin{equation}\label{eq:z evaluation at y^*}
        \real y^*(z)> \real y^*(Tx)-\varepsilon > 1-\varepsilon - \varepsilon = 1-2\varepsilon.
         \end{equation}
         Notice that $z$ can be represented as
         \[
         z = T\Big(\sum_{n=1}^k \lambda_nx_n\Big) = \sum_{n=1}^k\lambda_nT(x_n),
         \]
         where $x_n\in S(B_X,y^*,\varepsilon)$, $\lambda_n\geq 0$ for $n=1,\dots,k$, and $\sum_{n=1}^k \lambda_k=1$. It follows from \eqref{eq:z evaluation at y^*} that there exists $n_0\in\{1,\dots,k\}$ such that
         \[
         \real y^*\big(T(x_{n_0})\big) > 1-2\varepsilon.
         \]
         Since $x_{n_0}\in S(B_X,y^*,\varepsilon)$, we also have
         \[
         \real y^*\big(x_{n_0}+T(x_{n_0})\big)>2-3\varepsilon,
         \]
         and, consequently,
         \[
         \norm{\Id+T}\geq \norm{x_{n_0}+T(x_{n_0})} > \real y^*\big(x_{n_0}+T(x_{n_0})\big)>2-3\varepsilon.
         \]
         Since $\varepsilon$ was arbitrary, the desired result holds.
    \end{proof}

Let us give an interesting particular case, which extends \cite[Corollary~5.15]{aviles_slicely_2010}.

\begin{corollary}
Let $X$ be a Banach space with the Daugavet property and let $T\in \mathcal{L}(X,Y)$. Suppose that the set of strongly regular points of $\overline{T(B_X)}$ contains elements of norm arbitrarily close to $\|T\|$. Then, $T$ satisfies \eqref{eq:DE}. This happens, in particular, if
$T(B_X)\subset \overline{\conv}\bigl(\dent\bigl(\overline{T(B_X)}\bigr)\bigr)$.
\end{corollary}

\section*{Acknowledgements}
The authors thank M\"{a}rt P\~{o}ldvere for fruitful conversations on the paper's topic, which led to simplifications of the original proofs of Propositions \ref{proposition: 0 is always SCD in infinite absolute sum} and \ref{prop: SCD points in finite p-sum}.(b). They also thank Javier Mer\'{\i} for suggesting that it would be possible to find a Banach space $X$ such that $\SCD(B_X)=\{0\}$.

Part of this work was done during the visits of the first and second-named authors at the University of Granada in March and May 2023, for which they wish to express their gratitude for the warm welcome received.

The research of J.\ Langemets and M.\ L\~{o}o was supported by the Estonian Research Council grant PSG487.

The research of M.\ Mart\'in and A.\ Rueda Zoca was supported by MCIN/AEI/10.13039/501100011033 grant PID2021-122126NB-C31 and by Junta de Andaluc\'ia Grants FQM-0185.

The research of M.\ Mart\'{\i}n was also funded by ``Maria de Maeztu'' Excellence Unit IMAG, reference CEX2020-001105-M funded by MCIN/AEI/10.13039/501100011033.

The research of A.\ Rueda Zoca was also funded by Fundaci\'on S\'eneca ACyT Regi\'on de Murcia grant 21955/PI/22 and by Generalitat Valenciana project CIGE/2022/97.

\bibliographystyle{siam}

\end{document}